\documentclass[a4paper,11pt,reqno]{amsart}
\usepackage[utf8]{inputenc}
\usepackage{amsmath}
\usepackage{amssymb}
\usepackage{amsthm}
\usepackage{hyperref}
\usepackage{color}
\usepackage{a4wide}
\usepackage{dsfont}
\usepackage{mathdots}
\usepackage{tikz-cd}
\makeatletter

\newcommand{\fraka}{\mathfrak{a}}

\newcommand{\frakk}{\mathfrak{k}}

\newcommand{\frakn}{\mathfrak{n}}

\newcommand{\CC}{\mathbb{C}}

\newcommand{\NN}{\mathbb{N}}

\newcommand{\RR}{\mathbb{R}}
\newcommand{\ZZ}{\mathbb{Z}}
\newcommand{\calC}{\mathcal{C}}
\newcommand{\calD}{\mathcal{D}}

\newcommand{\calF}{\mathcal{F}}

\newcommand{\calL}{\mathcal{L}}
\newcommand{\calO}{\mathcal{O}}
\newcommand{\calP}{\mathcal{P}}

\newcommand{\0}{{\bf 0}}

\newcommand{\R}{\mathbb{R}}

\newcommand{\fedA}{\mathbf{A}}
\newcommand{\fedT}{\mathbf{T}}
\newcommand{\fedK}{\mathbf{K}}

\DeclareMathOperator{\GL}{GL}
\DeclareMathOperator{\SL}{SL}
\DeclareMathOperator{\upO}{O}
\DeclareMathOperator{\SO}{SO}

\DeclareMathOperator{\so}{\mathfrak{so}}
\DeclareMathOperator{\Ind}{Ind}
\DeclareMathOperator{\tr}{tr}

\DeclareMathOperator{\ad}{ad}
\DeclareMathOperator{\Ad}{Ad}

\DeclareMathOperator{\Hom}{Hom}

\DeclareMathOperator{\sgn}{sgn}
\DeclareMathOperator{\const}{const}
\DeclareMathOperator{\diag}{diag}

\renewcommand\Re{\operatorname{Re}}

\newcommand{\reg}{\textup{reg}}

\theoremstyle{plain}
\newtheorem{theorem}{Theorem}[section]
\newtheorem{proposition}[theorem]{Proposition}
\newtheorem{lemma}[theorem]{Lemma}
\newtheorem{corollary}[theorem]{Corollary}

\newtheorem{thmalph}{Theorem}

\theoremstyle{definition}

\newtheorem{remark}[theorem]{Remark}

\title[Symmetry breaking operators for $\big(\text{\normalfont GL}(n+1,\R),\text{\normalfont GL}(n,\R)\big)$]{Construction and analysis of symmetry breaking operators for the pair $\big(\text{\normalfont GL}(n+1,\R),\text{\normalfont GL}(n,\R)\big)$}

\author{Jonathan Ditlevsen \& Jan Frahm}
\address{Department of Mathematics, Aarhus University, Ny Munkegade 118, 8000 Aarhus C, Denmark}
\email{JonathanDitlevsen@gmail.com, Frahm@math.au.dk}

\begin{document}

\maketitle

\begin{abstract}
The pair of real reductive groups $(G,H)=(\GL(n+1,\RR),\GL(n,\RR))$ is a strong Gelfand pair, i.e. the multiplicities $\dim\Hom_H(\pi|_H,\tau)$ are either $0$ or $1$ for all irreducible Casselman--Wallach representations $\pi$ of $G$ and $\tau$ of $H$. This paper is concerned with the construction of explicit intertwining operators in $\Hom_H(\pi|_H,\tau)$, so-called \emph{symmetry breaking operators}, in the case where both $\pi$ and $\tau$ are principal series representations. Such operators come in families that depend meromorphically on the induction parameters, and we show how to normalize them in order to make the parameter dependence holomorphic. This is done by establishing explicit \emph{Bernstein--Sato} identities for their distribution kernels as well as explicit \emph{functional identities} for the composition of symmetry breaking operators with standard Knapp--Stein intertwining operators for $G$ and $H$. We also show that the obtained normalization is optimal and identify a subset of parameters for which the family of operators vanishes. Finally, we relate the operators to the local archimedean Rankin--Selberg integrals and use this relation to evaluate them on the spherical vectors.
\end{abstract}

\section*{Introduction}

\noindent One of the popular modern themes in the representation theory of reductive groups is the study of branching problems. Here, one is interested in the restriction $\pi|_H$ of an irreducible representation of a group $G$ to a subgroup $H$. In the category of Casselman--Wallach representations of real reductive groups, this leads to the study of the multiplicities
\[
\dim \Hom_H(\pi|_H,\tau)\in \NN\cup \{\infty\},
\]
quantifying how many times an irreducible representation $\tau$ of $H$ occurs as a quotient of $\pi|_H$. The study of these multiplicities has been a very active research topic during the last decades, leading to profound conjectures in particular in the case of multiplicity-one pairs $(G,H)$ (i.e. the multiplicities are either $0$ or $1$) such as the Gan--Gross--Prasad conjectures.

It turns out that for a number of applications such as branching laws for unitary representations~\cite{W21}, estimates for periods of automorphic forms~\cite{BR10,FS18}, partial differential equations~\cite{MOZ16,FOZ20} or conformal geometry~\cite{FOS19,JO20,JO22}, knowledge about the multiplicities is not sufficient and one needs explicit descriptions of intertwining operators in $\Hom_H(\pi|_H,\tau)$. This is also the objective of Stage C in Kobayashi's ABC program for branching problems (see \cite{K15}), where intertwining operators in $\Hom_H(\pi|_H,\tau)$ are called \emph{symmetry breaking operators}.

The construction and study of explicit symmetry breaking operators has been the subject of several works during the last decade. Most notably, in the two monographs~\cite{KS15,KS18} by Kobayashi--Speh symmetry breaking operators between principal series representations for the pair $(G,H)=(\upO(n+1,1),\upO(n,1))$ of rank one orthogonal groups are studied in great detail. For spherical principal series, this analysis was later extended by Frahm--Weiske~\cite{FW20} to all pairs $(G,H)$ where both $G$ and $H$ have real rank one and the multiplicities are finite. The only higher rank group $G$ for which similar results have been obtained is the product $G=\upO(1,n)\times\upO(1,n)$ of two rank one orthogonal groups for which intertwining operators for the diagonal subgroup $H=\upO(1,n)$ were studied by Clerc~\cite{C16}. In particular, all known examples are built from rank one groups.

In this paper, we attempt a detailed analysis of symmetry breaking operators between principal series for the multiplicity-one pair $(G,H)=(\GL(n+1,\RR),\GL(n,\RR))$ (see \cite{SZ12} for the multiplicity-one property). Note that $G$ has real rank $n+1$, so this work can be viewed as a first step towards the study of symmetry breaking operators for higher rank groups. We remark that many of our techniques have a chance of working in a more general setting, e.g. for multiplicity-one pairs of split groups, or even in the setting of strongly spherical reductive pairs (see \cite{F23} for the construction of explicit families of such operators).

\subsection*{Methods and results}
Let $(G,H)=(\GL(n+1,\RR),\GL(n,\RR))$ with $n\geq1$. Every pair $(\xi,\lambda)\in (\ZZ/2\ZZ)^{n+1}\times \CC^{n+1}$ determines a character $\chi_1\otimes\cdots\otimes\chi_{n+1}$ of the Borel subgroup $P_G$ of upper triangular matrices by letting the $i$-th diagonal entry act by the character
$$ \chi_i(x)=|x|_{\xi_i}^{\lambda_i}=\sgn(x)^{\xi_i}|x|^{\lambda_i} \quad \mbox{of }\RR^\times. $$
In the same way, $(\eta,\nu)\in(\ZZ/2\ZZ)^n\times\CC^n$ determines a character $\psi_1\otimes\cdots\otimes\psi_n$ of the Borel subgroup $P_H$ of $H$. We write
$$ \pi_{\xi,\lambda} = \Ind_{P_G}^G(\chi_1\otimes\cdots\otimes\chi_{n+1}) \qquad \mbox{and} \qquad \tau_{\eta,\nu} = \Ind_{P_H}^H(\psi_1\otimes\cdots\otimes\psi_n) $$ for the corresponding principal series representations of $G$ and $H$ (smooth normalized parabolic induction, see Section \ref{sec:PrincipalSeries} for details). Note that $\pi_{\xi,\lambda}$ and $\tau_{\eta,\nu}$ are irreducible for generic $\lambda$ and $\nu$, so the multiplicity $\dim\Hom_H(\pi_{\xi,\lambda}|_H,\tau_{\eta,\nu})$ is generically bounded above by $1$.

In \cite[\S7]{F23} the second author constructed a family of symmetry breaking operators $A_{\xi,\lambda}^{\eta,\nu}\in\Hom_H(\pi_{\xi,\lambda}|_H,\tau_{\eta,\nu})$ depending meromorphically on $(\lambda,\nu)\in\CC^{n+1}\times\CC^n$. This family is given by a distribution kernel $K_{\xi,\lambda}^{\eta,\nu}\in \calD'(G)$ in the sense that
\[
A_{\xi,\lambda}^{\eta,\nu}f(h)=\int_{G/P_G}K_{\xi,\lambda}^{\eta,\nu}(h^{-1}g)f(g)\,d(gP_G),
\]
where $d(gP_G)$ denotes the unique (up to scalar multiples) invariant integral of densities on $G/P_G$ and  the integral has to be understood in the distribution sense (see Section~\ref{sec:SymmetryBreakingKernels} for details). The kernel $K_{\xi,\lambda}^{\eta,\nu}$ can be expressed in terms of the minors
\[
    \Phi_k(g)=\det(g_{ij})_{i=n-k+2,\dots,n+1,\,j=1,\dots,k} \qquad \mbox{and} \qquad \Psi_k(g)=\det(g_{ij})_{i=n-k+1,\dots,n,\,j=1,\dots,k}
\]
by
\[
K_{\xi,\lambda}^{\eta,\nu}(g)=|\Phi_{n+1}(g)|^{\lambda_{n+1}+\frac{n}{2}}_{\xi_{n+1}}\prod_{i=1}^n|\Phi_i(g)|^{\lambda_i-\nu_{n+1-i}-\frac{1}{2}}_{\xi_i+\eta_{n+1-i}} |\Psi_i(g)|^{\nu_{n+1-i}-\lambda_{i+1}-\frac{1}{2}}_{\eta_{n+1-i}+\xi_{i+1}} \qquad (g\in G),
\]
whenever this defines a locally integrable function on $G$ (which is in particular the case if $\Re(\lambda_i-\nu_{n+1-i}-\frac{1}{2}),\Re(\nu_{n+1-i}-\lambda_{i+1}-\frac{1}{2})>0$ for all $i=1,\ldots,n$), and it follows from the results in \cite{F23} that it has a meromorphic extension to all $(\lambda,\nu)\in\CC^{n+1}\times\CC^n$. An important and natural problem in this context is to locate the poles of the kernel $K_{\xi,\lambda}^{\eta,\nu}$, as a function of $(\lambda,\nu)$, and to determine its residues in order to obtain explicit expressions of symmetry breaking operators for parameters where $K_{\xi,\lambda}^{\eta,\nu}$ is singular.

The first step in this program is to find a meromorphic function $n(\xi,\lambda,\eta,\nu)$ of $(\lambda,\nu)$ such that $\fedK_{\xi,\lambda}^{\eta,\nu}:=n(\xi,\lambda,\eta,\nu)^{-1}K_{\xi,\lambda}^{\eta,\nu}$ is holomorphic in $(\lambda,\nu)\in\CC^{n+1}\times\CC^n$. This is our first main result, and it is most easily stated in terms of local archimedean $L$-factors. Recall that the local $L$-factor of the character $\chi_{\varepsilon,\mu}(x)=|x|_\varepsilon^\mu$ of $\RR^\times$ ($\varepsilon\in\ZZ/2\ZZ$, $\mu\in\CC$) is defined by
\[
L(s,\chi_{\varepsilon,\mu}):=\pi^{-\frac{1}{2}(s+\mu+[\varepsilon])}\Gamma\Big(\frac{s+\mu+[\varepsilon]}{2}\Big) \qquad (s\in\CC),
\]
where $[\varepsilon]\in\{0,1\}$ denotes the unique representative of $\varepsilon\in\ZZ/2\ZZ$, and define
\[
n(\xi,\lambda,\eta,\nu) = \prod_{i+j\leq n+1} L(\tfrac{1}{2},\chi_i\psi_j^{-1})\prod_{i+j\geq n+2}L(\tfrac{1}{2},\chi_{i}^{-1}\psi_j).
\]

\begin{thmalph}[see Theorems~\ref{main thm} and \ref{optimal normalization}]\label{intro - hoved}
The family of distributions $\mathbf{K}_{\xi,\lambda}^{\eta,\nu}=n(\xi,\lambda,\eta,\nu)^{-1}K_{\xi,\lambda}^{\eta,\nu}$ depends holomorphically on $(\lambda,\nu)\in\CC^{n+1}\times \CC^n$ for all $(\xi,\eta)\in(\ZZ/2\ZZ)^{n+1}\times(\ZZ/2\ZZ)^n$. Moreover, this normalization is optimal in the sense that the set of $(\lambda,\nu)$ where $\mathbf{K}_{\xi,\lambda}^{\eta,\nu}=0$ is of codimension at least $2$ in $\CC^{n+1}\times\CC^n$.
\end{thmalph}

Note that if $\mathbf{K}_{\xi,\lambda}^{\eta,\nu}=0$ on a hyperplane given by the linear form $\ell(\xi,\lambda,\eta,\nu)=0$, then $\ell(\xi,\lambda,\eta,\nu)^{-1}\mathbf{K}_{\xi,\lambda}^{\eta,\nu}$ would still be holomorphic in $(\lambda,\nu)$ with possibly less zeros. In this sense, the second part of Theorem~\ref{intro - hoved} asserts that it is not possible to renormalize $\mathbf{K}_{\xi,\lambda}^{\eta,\nu}$ so that it vanishes for less parameters $(\lambda,\nu)$. In Theorem~\ref{Alle nuller} we show that the set of zeros is in fact of codimension equal to $2$ by identifying a locally finite union of affine subspaces of $\CC^{n+1}\times\CC^n$ of codimension $2$ on which $\mathbf{K}_{\xi,\lambda}^{\eta,\nu}$ vanishes.

The proof of Theorem~\ref{intro - hoved} uses essentially two main tools. The first one is a collection of explicit Bernstein--Sato type identities for the kernels $K_{\xi,\lambda}^{\eta,\nu}$. To find such identities, we make use of an idea due to Beckmann--Clerc~\cite{BC12} who considered a similar problem for the pair $(G,H)=(\upO(1,n)\times\upO(1,n),\operatorname{diag}\upO(1,n))$ (see also \cite{FOS19} for another instance of this technique for rank one orthogonal groups). Multiplying $K_{\xi,\lambda}^{\eta,\nu}$ with one of the functions $\Phi_i(g)$ or $\Psi_j(g)$ shifts some of the parameters, but in a direction in which the kernel becomes more regular. To shift parameters into directions in which the kernel is more singular, we compose the multiplication operator with standard Knapp--Stein intertwining operators for $G$ and $H$. A delicate and technical computation shows that this indeed produces differential operators and hence Bernstein--Sato identities. However, since the Bernstein--Sato identities do not take into account the sign characters specified by $\xi$ and $\eta$, this technique does not yield an optimal normalization in the above sense.

The second main tool fixes this. It consists of identities relating the composition of a symmetry breaking operator with a Knapp--Stein intertwining operator to another symmetry breaking operator. In \cite{KS15} such formulas are called \emph{functional identities}. To state these identities, recall that for every element $w$ in the Weyl group for $G$ there exists a meromorphic family of $G$-intertwining operators $T^w_{\xi,\lambda}:\pi_{\xi,\lambda}\to\pi_{w(\xi,\lambda)}$. Abusing notation, we use the same notation $T_{\eta,\nu}^w:\tau_{\eta,\nu}\to\tau_{w(\eta,\nu)}$ for the analogous $H$-intertwining operators parameterized by Weyl group elements $w$ for $H$ (see Section~\ref{sec:StdIntertwiners} for details).

\begin{thmalph}[see Theorem~\ref{main theorem}]\label{thm:IntroA}
The following identities of meromorphic functions in $(\lambda,\nu)\in\CC^{n+1}\times\CC^n$ hold:
\[
A_{\xi,\lambda}^{\eta,\nu}\circ T^w_{w^{-1}(\xi,\lambda)} = c_w(\xi,\lambda)A_{w^{-1}(\xi,\lambda)}^{\eta,\nu}\quad \text{and} \quad T^w_{\eta,\nu}\circ A_{\xi,\lambda}^{\eta,\nu}=d_w(\eta,\nu)A_{\xi,\lambda}^{w(\eta,\nu)},
\]
where $c_w(\xi,\lambda)$ and $d_w(\eta,\nu)$ are given explicitly in Theorem~\ref{main theorem}.
\end{thmalph}

The proof of Theorem~\ref{thm:IntroA} can be reduced to the case where $w$ is a simple reflection (see \eqref{Knapp-Stein}), in which case it boils down to an explicit computation using some identities for the minors $\Phi_i$ and $\Psi_j$. In Section~\ref{Eksempel} we discuss in detail the proofs of Theorems~\ref{intro - hoved} and \ref{thm:IntroA} in the case $n=1$. Most of the relevant arguments can already be seen here. In particular, we explain how these identities can be used to find the optimal normalization for $A_{\xi,\lambda}^{\eta,\nu}$ resp. $K_{\xi,\lambda}^{\eta,\nu}$.

When $\xi=(0,\ldots,0)$ and $\eta=(0,\ldots,0)$, both $\pi_\lambda=\pi_{\xi,\lambda}$ and $\tau_\eta=\tau_{\eta,\nu}$ are spherical, i.e. the contain a unique (up to scalar multiples) vector invariant under a maximal compact subgroup. We choose the maximal compact subgroups $K_G=\upO(n+1)$ and $K_H=\upO(n)$ and normalize the spherical vectors $\mathbf{1}_\lambda\in\pi_\lambda$ and $\mathbf{1}_\nu\in\tau_\nu$ such that they evaluate to $1$ at the identity. Applying the normalized operators $\mathbf{A}_\lambda^\nu=\mathbf{A}_{\xi,\lambda}^{\eta,\nu}$ to $\mathbf{1}_\lambda$ yields a spherical vector in $\tau_\nu$ which has to be a scalar multiple of $\mathbf{1}_\nu$. We determine the scalar explicitly:

\begin{thmalph}[see Theorem~\ref{thm:EvaluationSphericalVector}]\label{thm:IntroC}
For all $(\lambda,\nu)\in\CC^{n+1}\times\CC^n$ we have
\[
\mathbf{A}_{\lambda}^\nu\mathbf{1}_\lambda=e_G(\lambda)e_H(-\nu)\mathbf{1}_\nu,
\]
where $e_G$ resp. $e_H$ is the Harish-Chandra $e$-function for $G$ resp. $H$, i.e. the denominator in Harish-Chandra's $c$-function (see \eqref{eq:Efunction} for its definition).
\end{thmalph}

Instead of evaluating the integral for $\mathbf{A}_\lambda^\nu\mathbf{1}_\lambda$ directly, we relate the operators $A_{\xi,\lambda}^{\eta,\nu}$ to the local archimedean Rankin--Selberg integrals in Section~\ref{sec:EvaluationSphericalVector}, using recent results by Li--Liu--Su--Sun~\cite{LLSS23}. For $\GL(n+1)\times\GL(n)$, the local archimedean Rankin--Selberg integrals have been evaluated at the spherical vectors by Ishii--Stade~\cite{IS13}. This implies Theorem~\ref{thm:IntroC}. It also proves a conjecture by Frahm--Su~\cite[Conjecture 3.6]{FS21} for the special case $\GL(n+1)\times\GL(n)$. We further believe that the explicit relation between our operators $A_{\xi,\lambda}^{\eta,\nu}$ and the local Rankin--Selberg integrals has the potential to be used in the context of estimating Rankin--Selberg $L$-functions.

\subsection*{Outlook}

Having found the optimal normalization for $K_{\xi,\lambda}^{\eta,\nu}$, there are several natural follow-up questions. For instance, in Theorem~\ref{Alle nuller} we find a large set of $(\lambda,\nu)\in\CC^{n+1}\times\CC^n$ for which $\mathbf{K}_{\xi,\lambda}^{\eta,\nu}=0$. It is not clear to us how close this set is to the set of all zeros. A more refined question would be to determine the support of $\mathbf{K}_{\xi,\lambda}^{\eta,\nu}$ for all parameters. A first result in this direction is Proposition~\ref{prop:SupportOfK}.

In~\cite{W21} analogous operators for the pair $(G,H)=(\upO(1,n+1),\upO(1,n))$ were used to explicitly decompose the restriction of irreducible unitary representations of $G$ to $H$ into a direct integral of irreducible unitary representations of $H$. We plan to study similar questions for the pair $(G,H)=(\GL(n+1,\RR),\GL(n,\RR))$, first for $n=2$ and possibly also for $n>2$.

Finally, it would be interesting to see to which extent the methods used in this paper work for other pairs of groups $(G,H)$. We expect that for the multiplicity-one pairs $(\GL(n+1,\CC),\GL(n,\CC))$ rather analogous results can be achieved. Also, the multiplicity-one pairs of split orthogonal groups $(\upO(n+1,n),\upO(n,n))$ and $(\upO(n,n),\upO(n,n-1))$ have a good chance of behaving similarly, and so does the complex counterpart $(\upO(n+1,\CC),\upO(n,\CC))$. Some ideas might even work for general strongly spherical pairs (see \cite{F23}).

\subsection*{Notation}
We write $\mathds{1}$ for the vector having $1$'s in all entries. As usual, $e_i \in \RR^k$ denotes the $i$'th standard basis vector, $E_{i,j}$ is the matrix with zeros in all entries except the $(i,j)$-th entry which contains a one and $I_k$ is the $k\times k$ identity matrix. For $\mu\in\CC$ and $\varepsilon\in\ZZ/2\ZZ$ we write
$|x|^\mu_\varepsilon:=\sgn(x)^\varepsilon|x|^\mu$ ($x\in\RR^\times$). Moreover, for $\varepsilon\in\ZZ/2\ZZ$ we denote by $[\varepsilon]$ its unique representative in $\{0,1\}$. We write $\NN=\{1,2,\dots\}$ and $\NN_0=\NN\cup \{0\}$.

\subsection*{Acknowledgements}
Both authors were supported by a research grant from the Villum Foundation (Grant No. 00025373).

\section{Principal series representations and Knapp--Stein intertwining operators for general linear groups} \label{generelt PRR}

\noindent In this section we define principal series representations of $G=\GL(k,\RR)$ and discuss intertwining operators between them.

\subsection{Principal series representations}\label{sec:PrincipalSeries}

We fix the minimal parabolic subgroup $P\subseteq G$ of upper triangular matrices. Its unipotent radical $N$ is the subgroup of unipotent upper triangular matrices, i.e. those with diagonal entries equal to $1$. The diagonal matrices form a Levi subgroup of $P$ which we write as $MA$ with $M$ the diagonal matrices with $\pm1$ on the diagonal and $A$ the diagonal matrices with positive diagonal entries. Write $\fraka$ and $\frakn$ for the Lie algebras of $A$ and $N$.

The irreducible representations of $M$ are one-dimensional and given by the characters
\begin{equation*} M\rightarrow \{-1,1\}, \quad \text{diag}(\varepsilon_1,\dots,\varepsilon_{k}) \mapsto \varepsilon_1^{\xi_1}\dotsb \varepsilon_{k}^{\xi_{k}},\end{equation*}
where $\xi=(\xi_1,\dots,\xi_{k})\in(\ZZ/2\ZZ)^k$. We identify $\xi$ with the corresponding character of $M$. Furthermore, we make the identification $\mathfrak{a}_{\CC}^*\simeq \CC^{k}$ by mapping $\lambda\mapsto \big(\lambda(E_{1,1}),\dots, \lambda (E_{k,k})\big)$. Then $\rho=\frac{1}{2}\tr\ad|_\frakn$ corresponds to $\frac{1}{2}(k-1,k-3,\dots,3-k,1-k)$. For $\lambda\in\CC^k$ we write $e^\lambda$ for the character of $A$ given by $e^\lambda(e^H)=e^{\lambda(H)}$ ($H\in\fraka$).

For $\xi\in(\ZZ/2\ZZ)^k$ and $\lambda\in\fraka_\CC^*$ we extend the character $\xi\otimes e^\lambda$ of $MA$ trivially to $P=MAN$ and write $\xi\otimes e^\lambda\otimes1$ for it. Using smooth parabolic induction from $P$ to $G$ we obtain the principal series representation $\pi_{\xi,\lambda}=\Ind_P^G(\xi\otimes e^\lambda\otimes 1)$ as the left-regular representation of $G$ on 
\[
\{f\in C^\infty(G)\,|\,f(gman)=\xi(m)^{-1}a^{-\lambda-\rho}f(g)\,\,\forall man\in MAN\}.
\]

The representation $\pi_{\xi,\lambda}$ is irreducible if for all $1\leq i,j\leq k$ (see \cite{S77})
$$ \lambda_i-\lambda_j\not\in\begin{cases}2\ZZ+1&\mbox{if $\xi_i=\xi_j$,}\\2\ZZ\setminus\{0\}&\mbox{if $\xi_i\neq\xi_j$.}\end{cases} $$

We fix the maximal compact subgroup $K=\upO(n)$ and write $\frakk=\so(n)$ for its Lie algebra. The Weyl group $W:=N_K(A)/Z_K(A)$ can be identified with the symmetric group $\mathbb{S}_k$ in terms of permutation matrices. It acts on $\widehat{M}$ by $[w\xi](m)=\xi(\tilde{w}^{-1}m\tilde{w})$ ($m\in M$) where $w=[\tilde{w}]\in W$ and similarly on $\fraka_\CC^*$ by $[w\lambda](H)=\lambda(\Ad(\tilde{w})^{-1}H)$ ($H\in\fraka_\CC^*$). Identifying $w\in W$ with the corresponding permutation, the action on $\widehat{M}\simeq(\ZZ/2\ZZ)^k$ and $\fraka_\CC^*\simeq\CC^k$ is given by permuting the entries of vectors in $(\ZZ/2\ZZ)^k$ and $\CC^k$. Abusing notation we do not distinguish between elements of $W$, permutations and the corresponding permutation matrices (i.e. $we_i=e_{w(i)}$).

There exists a unique (up to scalar multiples) $G$-invariant integral $d(gP)$ on the space of sections of the bundle $G\times_P\CC_{2\rho}\to G/P$, where $\CC_{2\rho}=1\otimes e^{2\rho}\otimes1$ is the one-dimensional representation corresponding to the modular function of $P$ (see e.g. \cite[Section 1]{CF24} for a thorough discussion). This integral gives rise to a $G$-invariant pairing
\begin{equation}
    \pi_{\xi,\lambda}\times\pi_{\xi,-\lambda}\to\CC, \quad (f_1,f_2)\mapsto\int_{G/P}f_1(g)f_2(g)\,d(gP).\label{eq:InvariantPairing}
\end{equation}
Using the Iwasawa decomposition $G=KAN$ or the open dense Bruhat cell $\overline{N}MAN$ with $\overline{N}$ the opposite unipotent radical consisting of lower triangular unipotent matrices, this pairing can also be written as
\begin{equation}
    \int_{G/P}f_1(g)f_2(g)\,d(gP) = \int_{K}f_1(k)f_2(k)\,dk = \int_{\overline{N}}f_1(\overline{n})f_2(\overline{n})\,d\overline{n},\label{eq:GinvariantIntegralKNbar}
\end{equation}
where $dk$ and $d\overline{n}$ are suitably normalized Haar measures on $K$ and $\overline{N}$ (cf. \cite[formulas (5.17) and (5.25)]{K16}).

\subsection{Knapp--Stein intertwining operators}\label{sec:StdIntertwiners}

For every $w=[\tilde{w}]\in W$ and $f\in\pi_{\xi,\lambda}$ the integral 
\[
T_{\xi,\lambda}^wf(g)=\int_{\overline{N}\cap \tilde{w}^{-1}N\tilde{w}}f(g\tilde{w}\overline{n})\,d\overline{n} \qquad (g\in G),
\]
converges absolutely if $\Re(\lambda_i)>\Re(\lambda_i)$ whenever $w(i)>w(j)$, $i<j$, and defines an intertwining operator
$$ T_{\xi,\lambda}^w:\pi_{\xi,\lambda}\to \pi_{w\xi,w\lambda} $$
known as the Knapp--Stein intertwining operator (see e.g. \cite[Proposition 7.8]{K16}). It can be shown that the family of operators $T_{\xi,\lambda}^w$ has a meromorphic extension to all $\lambda\in\fraka_\CC^*$ (if viewed as an intertwining operator in the compact picture where the representation space is independent of $\lambda$). Whenever defined, they satisfy
\begin{equation}\label{Knapp-Stein}
	T^{w'w}_{\xi,\lambda}=T^{w'}_{w\xi,w\lambda}\circ T^{w}_{\xi,\lambda}
\end{equation}
for all $w,w'\in W$ with $\ell(w'w)=\ell(w')+\ell(w)$, where $\ell$ denotes the length of an element in $W$ (see e.g. \cite[Proposition 7.10]{K16}). Moreover, with respect to the $G$-invariant pairing of $\pi_{\xi,\lambda}$ and $\pi_{\xi,-\lambda}$ discussed in Section~\ref{sec:PrincipalSeries}, the adjoint of $T_{\xi,\lambda}^w$ is given by $T_{w(\xi,-\lambda)}^{w^{-1}}$ (see \cite[Proposition 14.10]{K16}).

We therefore have a closer look at the special case of simple transpositions. For $i=1,\dots,k-1$, let $w_i$ denote the simple transposition swapping $i$ and $i+1$, considered as a permutation matrix. Then $\overline{N}\cap w_i^{-1}Nw_i=\{\overline{n}_i(x):x\in\RR\}$, where $\overline{n}_i(x)$ is the $\GL(2,\RR)$ element
\[
\begin{pmatrix}
    1 & 0\\
    x & 1
\end{pmatrix}
\]
in $\GL(k,\RR)$ using the embedding
\[
\Bigg(\begin{array}{c|c|c}
	I_{i-1} & & \\
	\hline
	&GL(2,\RR) & \\
	\hline 
	& & I_{k-i-1}
\end{array}\Bigg).
\]
Using the $\GL(2,\RR)$ computation 
\[
\begin{pmatrix}
	0 & 1\\
	1 & 0
\end{pmatrix}
\begin{pmatrix}
	1& 0\\
	x & 1
\end{pmatrix}
=
\begin{pmatrix}
	x & 1\\
	1 & 0
\end{pmatrix}=
\begin{pmatrix}
	1 & 0\\
	\frac{1}{x} & 1
\end{pmatrix}
\begin{pmatrix}
	x & 0\\
	0 &-\frac{1}{x}
\end{pmatrix}
\begin{pmatrix}
	1 & \frac{1}{x}\\
	0 & 1
\end{pmatrix}\qquad (x\neq 0),
\]
we can decompose $w_i\overline{n}_i(x)$ according to the decomposition $\overline{N}MAN$. This allows us to get a more explicit formula for the Knapp--Stein intertwiner:
\begin{equation}
		T^{w_i}_{\xi,\lambda}f(g)
		=\int_\RR f(gw_i\overline{n}_i(x))\,dx
		=(-1)^{\xi_{i+1}}\int_{\R}|x|^{\lambda_{i}-\lambda_{i+1}-1}_{\xi_i+\xi_{i+1}}f(g\overline{n}_i(x))\,dx,\label{eq:SimpleTranspositionKSonNbar}
\end{equation}
where we have applied the change of variables $x\to x^{-1}$ in the last step. This integral has simple poles as a function of $\lambda_i$ and $\lambda_j$ and we therefore renormalize it by dividing by a meromorphic function with simple poles at the same places. This is most easily expressed in terms of archimedean local $L$-factors. Using the notation $\chi_{\varepsilon,\mu}(x)=|x|_\varepsilon^\mu=\sgn(x)^\varepsilon|x|^\mu$ ($x\in\RR^\times$) for a character $\chi_{\varepsilon,\mu}$ of $\RR^\times$ ($\varepsilon\in\ZZ/2\ZZ$, $\mu\in\CC$), the archimedean local $L$-factor $L(s,\chi_{\varepsilon,\mu})$ is defined by
\begin{equation}
    L(s,\chi_{\varepsilon,\mu}) = \pi^{-\frac{s+\mu+[\varepsilon]}{2}}\Gamma\left(\frac{s+\mu+[\varepsilon]}{2}\right) \qquad (s\in\CC),\label{eq:DefLocalLFactor}
\end{equation}
where $[\varepsilon]$ denotes the unique representative of $\varepsilon\in\ZZ/2\ZZ$ in $\{0,1\}$. It is a meromorphic function in $s\in\CC$ and also depends meromorphically on the character $\chi_{\varepsilon,\mu}\in\Hom_{\textup{grp}}(\RR^\times,\CC)$. Putting $\chi_i=\chi_{\xi_i,\lambda_i}$, we define the renormalized intertwining operators by
\[
\mathbf{T}_{\xi,\lambda}^{w_i}=\frac{1}{L(0,\chi_i\chi_{i+1}^{-1})}T_{\xi,\lambda}^{w_i}.
\]
With this renormalization the family $\mathbf{T}_{\xi,\lambda}^{w_i}$ becomes holomorphic in $\lambda\in\CC^k$ and nowhere vanishing (follows for instance from \cite[Proposition 7]{M97} or \cite[Prop 2.3]{BD23}). For a general $w\in W$ we normalize $T^w_{\xi,\lambda}$ by writing $w=w_{i_1}\cdots w_{i_r}$ with $r=\ell(w)$ and then using the normalization for each $T^{w_i}_{\xi,\lambda}$ in \eqref{Knapp-Stein}.

Similarly we can decompose $\overline{n}_i(x)=k(x)a(x)n(x)$ with $k(x)\in K$, $a(x)\in A$ and $n(x)\in N$ by the $\GL(2,\RR)$-computation 
\begin{equation} \label{iwasawa}
	\begin{pmatrix}
		1 & 0 \\
		x & 1
	\end{pmatrix}=\bigg [\frac{1}{\sqrt{1+x^2}}\begin{pmatrix}
		1 & -x\\
		x & 1
	\end{pmatrix}\bigg ]\begin{pmatrix}
		\sqrt{1+x^2} & 0\\
		0& \frac{1}{\sqrt{1+x^2}}
	\end{pmatrix}
\begin{pmatrix}
		1 & \frac{x}{1+x^2}\\
		0 & 1
	\end{pmatrix}.
\end{equation}
This gives the following alternative expression for the intertwining operator $T_{\xi,\lambda}^{w_i}$:
\begin{align*}
	T^{w_i}_{\xi,\lambda}f(g)&=\int_\RR (1+x^2)^{\frac{\lambda_{i+1}-\lambda_i-1}{2}}f(gw_ik(x))\,dx.	
\end{align*}

\section{Symmetry breaking operators and distribution kernels} \label{PSR for GL}

\noindent In this section we recall the description of symmetry breaking operators between principal series representations of $G=\GL(n+1,\RR)$ and $H=\GL(n,\RR)$ in terms of distribution kernels and introduce the explicit family of operators/kernels that is the main object studied in this paper.

\subsection{Symmetry breaking kernels}\label{sec:SymmetryBreakingKernels}

The pair $(G,H)=(\GL(n+1,\R),\GL(n,\R))$ is a strong Gelfand pair, i.e. $\dim\Hom_H(\pi|_H,\tau)\leq1$ for all irreducible smooth admissible Fréchet representations $\pi$ of $G$ and $\tau$ of $H$ of moderate growth (see \cite[Theorem B]{SZ12}). Elements of the space $\Hom_H(\pi|_H,\tau)$ are referred to as \emph{symmetry breaking operators}.

Writing $\pi_{\xi,\lambda}$ ($\xi\in(\ZZ/2\ZZ)^{n+1}$, $\lambda\in\CC^{n+1}$) for the principal series representations of $G$ and $\tau_{\eta,\nu}$ ($\eta\in(\ZZ/2\ZZ)^n$, $\nu\in\CC^n$) for the principal series representations of $H$ as in Section~\ref{generelt PRR}, the space $\Hom_H(\pi_{\xi,\lambda}|_H,\tau_{\eta,\nu})$ is therefore at most one-dimensional for generic $\lambda$ and $\nu$. 

Denote by $P_G=M_GA_GN_G$ and $P_H=M_HA_HN_H$ the minimal parabolic subgroups pf $G$ and $H$ defined in Section~\ref{generelt PRR}. Following \cite{F23}, we identify
$$ \Hom_H(\pi_{\xi,\lambda}|_H,\tau_{\eta,\nu}) \simeq \calD'(G)_{\xi,\lambda}^{\eta,\nu}, $$
where
\begin{multline*}
	\calD'(G)_{\xi,\lambda}^{\eta,\nu} = \{K\in\calD'(G):K(m_Ha_Hn_Hgm_Ga_Gn_G)=\xi(m_G)\eta(m_H)a_G^{\lambda-\rho_G}a_H^{\nu+\rho_H}K(g)\\\mbox{for all }m_Ga_Gn_G\in P_G,m_Ha_Hn_H\in P_H,g\in G\}
\end{multline*}
in the sense that a distribution $K\in\calD'(G)_{\xi,\lambda}^{\eta,\nu}$ defines a symmetry breaking operator $A\in\Hom_H(\pi_{\xi,\lambda}|_H,\tau_{\eta,\nu})$ by
\begin{equation}
    Af(h) = \int_{G/P_G} K(h^{-1}g)f(g)\, d(gP_G) \qquad (f\in\pi_{\xi,\lambda},h\in H).\label{eq:SBOinTermsOfKernel}
\end{equation}
Here, the integral has to be understood in the distribution sense using the $G$-invariant integral $d(gP_G)$ on sections of the bundle $G\times_{P_G}\CC_{2\rho_G}\to G/P_G$ (see Section~\ref{sec:PrincipalSeries}). Note that by the equivariance of $K$ and $f$, the distribution $g\mapsto K(h^{-1}g)f(g)$ is indeed a distributional section of this bundle. Note that by \eqref{eq:GinvariantIntegralKNbar} we can also write
\begin{equation}
    Af(h) = \int_{K_G} K(h^{-1}k)f(k)\,dk = \int_{\overline{N}_G} K(h^{-1}\overline{n})f(\overline{n})\,d\overline{n} \qquad (f\in\pi_{\xi,\lambda},h\in H),\label{eq:SBOasIntegral}
\end{equation}
again interpreting integrals in the distribution sense.

\subsection{An explicit family of kernels} \label{sec: kernel}

\noindent Following \cite{F23}, we consider for $1\leq p\leq n+1$ and $1\leq q\leq n$ the functions $$\widetilde{\Phi_p}(g)=\det\big ((g_{ij})_{1\leq i,j\leq p})\quad \text{and}\quad \widetilde{\Psi}_q(g)=\det\big( (g_{ij})_{2\leq i\leq q+1, 1\leq j\leq q}\big ) \qquad (g\in G). $$
Let
$$w_0=\begin{pmatrix}
	& & 1\\
	& \iddots &\\
	1 & & 
\end{pmatrix}, $$
be a representative of the longest Weyl group element of $G$ and set $\Phi_p(g)=\widetilde{\Phi}_p(w_0g)$ and $\Psi_q(g)=\widetilde{\Psi}_q(w_0g)$. 
Now consider the kernel 
$$K_{\xi,\lambda}^{\eta,\nu}(g)=|\Phi_1(g)|_{\delta_1}^{s_1}\dotsb|\Phi_{n+1}(g)|_{\delta_{n+1}}^{s_{n+1}}|\Psi_1(g)|_{\varepsilon_1}^{t_1}\dotsb |\Psi_n(g)|_{\varepsilon_n}^{t_n},\qquad (g\in G)$$
where $s_i=\lambda_i-\nu_{n+1-i}-\frac{1}{2}$, $t_i=\nu_{n+1-i}-\lambda_{i+1}-\frac{1}{2}$ for $i=1,\dots,n$ and $s_{n+1}=\lambda_{n+1}+\frac{n}{2}$. Likewise $\delta_i=\xi_i-\eta_{n+1-i}$, $\varepsilon_i=\eta_{n+1-i}-\xi_{i+1}$ for $i=1,\dots,n$ and $\delta_{n+1}=\xi_{n+1}$. The exponents $(s,t)$ are related to $(\lambda,\nu)\in \CC^{n+1}\times\CC^n$ by an invertible affine linear coordinate transformation given by
$$ \begin{pmatrix}
	s_1\\
	t_1\\
	\vdots\\
	t_n\\
	s_{n+1}
\end{pmatrix}=\begin{pmatrix}
	1 & -1 & 0& \dotsb &0 & 0\\
	0 & 1 &-1 &\dotsb & 0 & 0\\
	\vdots & \vdots &\ddots &\ddots & \vdots&\vdots \\
	0 & 0 & 0 & \dotsb & 1 &-1\\
	0 & 0 & 0 & \dotsb &0 & 1
\end{pmatrix} \begin{pmatrix}
	\lambda_1\\
	\nu_n\\
	\lambda_2\\
	\vdots\\
	\lambda_{n+1}
\end{pmatrix}-\frac{1}{2}\begin{pmatrix}
	1\\
	1\\
	\vdots\\
	1\\
	-n
\end{pmatrix}$$
with inverse 
$$
\begin{pmatrix}
	\lambda_1\\
	\nu_n\\
	\lambda_2\\
	\vdots\\
	\lambda_{n+1}
\end{pmatrix}
=\begin{pmatrix}
	1 & 1 & 1& \dotsb &1 & 1\\
	0 & 1 &1 &\dotsb & 1 & 1\\
	\vdots & \vdots &\ddots &\ddots & \vdots&\vdots \\
	0 & 0 & 0 & \dotsb & 1 &1\\
	0 & 0 & 0 & \dotsb &0 & 1
\end{pmatrix} \begin{pmatrix}
	s_1\\
	t_1\\
	\vdots\\
	t_n\\
	s_{n+1}
\end{pmatrix} +\frac{1}{2}\begin{pmatrix}
	n\\
	n-1\\
	\vdots\\
	-(n-1)\\
	-n
\end{pmatrix}.$$
The same coordinate transformation relates $(\xi,\eta)$ and $(\delta,\varepsilon)$ if we disregard the affine part. We will refer to the parameters $(\xi,\lambda,\eta,\nu)$ as \emph{principal series parameters} and to $(\delta,s,\varepsilon,t)$ as \emph{spectral parameters}.

\begin{proposition}
    The function $K_{\xi,\lambda}^{\eta,\nu}$ is locally integrable on $G$ if and only if $\Re(\lambda_i-\nu_j+\frac{1}{2})>0$ for all $i+j\leq n+1$ and $\Re(\nu_j-\lambda_i+\frac{1}{2})>0$ for all $i+j>n+1$.
\end{proposition}

\begin{proof}
    In Proposition~\ref{prop:ComparisonOvsSBO} we relate the kernel $K_{\xi,\lambda}^{\eta,\nu}$ to an invariant form on $\pi_{\xi,\lambda}\otimes\tau_{\eta,-\nu}$ defined by an integral over $H$. This integral was studied in \cite{LLSS23}, and \cite[Proposition 1.4]{LLSS23} shows that the above condition on $(\lambda,\nu)$ is sufficient for its convergence. That the condition is also necessary follows from Theorem~\ref{thm:EvaluationSphericalVector} and the explicit gamma factors used in the normalization. (Note that only the sufficient condition is used in what follows.)
\end{proof}

A short computation shows that the kernel $K_{\xi,\lambda}^{\eta,\nu}$ has the desired equivariance properties, so it belongs to $\calD'(G)_{\xi,\lambda}^{\eta,\nu}$ whenever its locally integrable. In this region, $K_{\xi,\lambda}^{\eta,\nu}$ depends analytically on $s_i$ and $t_i$ (hence on $\lambda$ and $\nu$), and it was shown in \cite[Section 7]{F23} that it extends to a family of distributions depending meromorphically on $(\lambda,\nu)\in\CC^{n+1}\times\CC^n$. We write $A_{\xi,\lambda}^{\eta,\nu}: \pi_{\xi,\lambda}|_H\to \tau_{\eta,\nu}$ for the corresponding meromorphic family of intertwining operators. Abusing notation, we sometimes write $K_{\delta,s}^{\varepsilon,t}$ and $A_{\delta,s}^{\varepsilon,t}$ instead of $K_{\xi,\lambda}^{\eta,\nu}$ and $A_{\xi,\lambda}^{\eta,\nu}$. The aim of the following sections is to find a normalization for $A_{\xi,\lambda}^{\eta,\nu}$ so that it becomes holomorphic in $(\lambda,\nu)\in\CC^{n+1}\times\CC^n$.

\section{The case \texorpdfstring{$n=1$}{n=1}}\label{Eksempel}

\noindent In this section we discuss the proof of analytical continuation in the case where $n=1$, that is where $(G,H)=(\GL(2,\RR), \GL(1,\RR))$. This case is covered in the general proof in Sections \ref{Section funktionalligning}, \ref{section:BS-identities} and \ref{Section: analytic extension} but it showcases the main ideas avoiding some technical details.

Let $n=1$, then the Weyl group of $G$ consists of two elements $1$ and $w$. For the principal series representation $\pi_{\xi,\lambda}$ with parameters $\xi=(\xi_1,\xi_2)\in (\ZZ/2\ZZ)^2$ and $\lambda=(\lambda_1,\lambda_2)\in \CC^2$ we have the intertwining operator $T_{\xi,\lambda}^w:\pi_{\xi,\lambda}\to \pi_{w(\xi,\lambda)}$ given by 
\[
T_{\xi,\lambda}^wf(g)=(-1)^{\xi_2}\int_\RR |x|^{\lambda_1-\lambda_2-1}_{\xi_1+\xi_2}f\big( g\big (\begin{smallmatrix}
    1 & 0\\
    x  & 1
\end{smallmatrix}\big)\big) dx,\quad (g\in G),
\]
where $w\xi=(\xi_2,\xi_1)$ and $w\lambda=(\lambda_2,\lambda_1)$. As $H=\GL(1,\RR)=\RR^\times$, the principal series representations $\tau_{\eta,\nu}$ are one-dimensional and depends on $\eta \in \ZZ/2\ZZ$ and $\nu \in \CC$ with no standard intertwining operators between them. The space of equivariant distributional kernels $\calD'(G)_{\xi,\lambda}^{\eta,\nu}$ is generically one-dimensional and spanned by the distribution
\[
K_{\xi,\lambda}^{\eta,\nu}(g)=|\Phi_1(g)|^{s_1}_{\delta_1}|\Psi_1(g)|^{t_1}_{\varepsilon_1}|\Phi_2(g)|^{s_2}_{\delta_2}=|g_{21}|^{\lambda_1-\nu-\frac{1}{2}}_{\xi_1+\eta}|g_{11}|^{\nu-\lambda_2-\frac{1}{2}}_{\eta+\xi_2}|\det(g)|^{\lambda_2+1}_{\xi_2}.
\]
This kernel is locally integrable whenever $\Re(s_1)>-1$, $\Re(t_1)>-1$. The goal is to analytically extend it as a distribution to all $(s,t)\in \CC^2\times \CC$. To ease notation we suppress $(\xi,\eta)$ and $(\delta,\varepsilon)$ in the notation. We focus on the analytic extension in the parameter $t=t_1$. For this, we are searching for so-called \emph{Bernstein--Sato identities} which consist of a differential operator $\calD(s,t)$ on $G$ depending holomorphically on $(s,t)$ and a polynomial $p(s,t)$ such that 
\[
\calD(s,t)K_{s_1,s_2}^t=p(s,t)K_{s_1,s_2+1}^{t-1}.
\]

\subsection{Conjugation of multiplication operators by Knapp--Stein operators}

Consider the multiplication map $M$ which maps a distribution $u$ on $G$ to $\Phi_1\cdot u$. We can consider this as a map on $\calD'(G)_{\lambda_1,\lambda_2}^{\nu}$ by
\[
M:\calD'(G)_{\lambda_1,\lambda_2}^{\nu}\to \calD'(G)_{\lambda_1+1,\lambda_2}^{\nu},
\]
since $\Phi_1(g)=g_{21}$ has the equivariance properties $\Phi_1(hgman)=x_1\Phi_1(g)$ for $ma=\diag(x_1,x_2)$, $n\in N_G$ and $h\in H$.
For $\Re(s_1),\Re(t)>-1$ the function $g\mapsto K_{\lambda_1,\lambda_2}^{\nu}(g)$  can be considered as a distribution vector in $\pi_{-\lambda}$, and by 
extending $T_{-\lambda}^w$ to distribution vectors by duality we obtain a map
\begin{equation}
    T_{-\lambda}^w:\calD'(G)_{\lambda_1,\lambda_2}^{\nu}\to \calD'(G)_{\lambda_2,\lambda_1}^{\nu}.\label{eq:KSonSBOkernelsForN=1}
\end{equation}
We get a new map $\calD(\lambda,\nu):\calD'(G)_{\lambda_1,\lambda_2}^\nu\to \calD'(G)_{\lambda_1,\lambda_2+1}^\nu$ by composing these maps in the following way
\[
\begin{tikzcd}
{\calD'(G)_{\lambda_1,\lambda_2}^{\nu}} \arrow[d, "{T_{-\lambda}^w}"'] \arrow[r, dashed, "{\calD(\lambda,\nu)}"] & {\calD'(G)_{\lambda_1,\lambda_2+1}^{\nu}}                                                \\
{\calD'(G)_{\lambda_2,\lambda_1}^{\nu}} \arrow[r, "M"']                        & {\calD'(G)_{\lambda_2+1,\lambda_1}^{\nu}}. \arrow[u, "{T_{-\lambda_2-1,-\lambda_1}^w}"']
\end{tikzcd}
\]
Note that the transformation $(\lambda_1,\lambda_2,\nu)\mapsto(\lambda_1,\lambda_2+1,\nu)$ corresponds to $(s_1,s_2,t)\mapsto(s_1,s_2+1,t-1)$, so while the multiplication map $M$ increases the spectral parameter $s_1$ by one, its composition with Knapp--Stein operators decreases the spectral parameter $t$ by one and hence makes it possible to analytically extend from $\{\Re(t)>0\}$ to $\{\Re(t)>-1\}$. Abusing notation, we also write $\calD(s,t)$ for $\calD(\lambda,\nu)$.

Since $\calD'(G)_{\lambda_1,\lambda_2}^\nu$ is generically of dimension one there exists a meromorphic function $p_{\calD}(s,t)$ such that 
\[
\calD(s,t)K_{s_1,s_2}^t=p_{\calD}(s,t)K_{s_1,s_2+1}^{t-1}.
\]
We now show that $\calD(s,t)$  is in fact a differential operator and $p_{\calD}(s,t)$ is a polynomial, which means we have established the Bernstein--Sato identity. 

\subsection{Computation of the differential operator} \label{eksempel diff operator}

We start out by computing the composition $T_{-\lambda_2-1,-\lambda_1}^w\circ M$ using \eqref{eq:SimpleTranspositionKSonNbar}:
\begin{align*}\label{induktionsstart}
    [T_{-\lambda_2-1,-\lambda_1}^w\circ Mf](g)&=(-1)^{\xi_1}\int_\RR |x|^{\lambda_1-\lambda_2-2}_{\xi_1+\xi_2+1}(\Phi_1f)\big( g\big (\begin{smallmatrix}
    1 & 0\\
    x  & 1
\end{smallmatrix}\big)\big)\,dx\\
&=(-1)^{\xi_1}\int_\RR |x|^{\lambda_1-\lambda_2-2}_{\xi_1+\xi_2+1}(g_{21}+xg_{22})f\big ( g\big (\begin{smallmatrix}
    1 & 0\\
    x  & 1
\end{smallmatrix}\big)\big)\,dx.
\end{align*}
If we split the integral into the two terms coming from $g_{21}+xg_{22}$, we see that the term containing $xg_{22}$ is just $g_{22}T_{-w\lambda}^wf(g)$. For the term containing $g_{21}$ we write
$$ |x|^{\lambda_1-\lambda_2-2}_{\xi_1+\xi_2+1}=(\lambda_1-\lambda_2-1)^{-1}\frac{\partial}{\partial x}|x|^{\lambda_1-\lambda_2-1}_{\xi_1+\xi_2}, $$
use integration by parts and $\frac{\partial}{\partial x}f\big ( g\big (\begin{smallmatrix}
    1 & 0\\
    x  & 1
\end{smallmatrix}\big)\big)=(g_{12}\partial_{g_{11}}+g_{22}\partial_{g_{21}})f\big ( g\big (\begin{smallmatrix}
    1 & 0\\
    x  & 1
\end{smallmatrix}\big)\big)$ and find that 
\begin{align*}
(-1)^{\xi_1}g_{21}\int_{\RR}|x|^{\lambda_1-\lambda_2-2}_{\xi_1+\xi_2+1}f\big ( g\big (\begin{smallmatrix}
    1 & 0\\
    x  & 1
\end{smallmatrix}\big)\big)\,dx&=(-1)^{\xi_1+1}\frac{g_{21}}{\lambda_1-\lambda_2-1}\int_\RR |x|^{\lambda_1-\lambda_2-1}_{\xi_1+\xi_2}\frac{d}{dx}f\big ( g\big (\begin{smallmatrix}
    1 & 0\\
    x  & 1
\end{smallmatrix}\big)\big)\,dx\\
&=-\frac{g_{21}(g_{12}\partial_{g_{11}}+g_{22}\partial_{g_{21}})}{\lambda_1-\lambda_2-1}T_{-w\lambda}^wf(g).
\end{align*}
Thus, if we slightly change the definition of $\calD(\lambda,\nu)$ to be $(\lambda_1-\lambda_2-1)T_{-\lambda_2-1,-\lambda_1}^w\circ M\circ (T_{-w\lambda}^w)^{-1}$, we see that it is given by
\[
\calD(\lambda,\nu)=(\lambda_1-\lambda_2-1)g_{22}-g_{21}(g_{12}\partial_{g_{11}}+g_{22}\partial_{g_{21}}),
\]
which indeed is a differential operator depending holomorphically on $\lambda$ and $\nu$.

\subsection{Finding the polynomial}\label{eksempel funktionalligning}

One possible way to obtain the polynomial $p_{\calD}(s,t)$ in the Bernstein--Sato identity would be to simply apply $\calD(\lambda,\nu)$ to $K_{\lambda_1,\lambda_2}^\nu$ in the range of parameters where $K_{\lambda_1,\lambda_2}^\nu$ is sufficiently many times differentiable. An alternative way is to compute the constant $d(\lambda,\nu)$ in the identity
\begin{equation}
    T_{-\lambda}^wK_{\lambda_1,\lambda_2}^\nu=d(\lambda,\nu)K_{\lambda_2,\lambda_1}^\nu\label{eq:FunctionalEqForN=1}
\end{equation}
(which holds by \eqref{eq:KSonSBOkernelsForN=1} and the property that generically $\dim\calD'(G)_{\lambda_1,\lambda_2}^\nu=1$) and then use the definition of $\calD(\lambda,\nu)$ given in terms of composition of $M$ with the Knapp--Stein intertwiners. In view of \eqref{eq:SimpleTranspositionKSonNbar} we have
\begin{align*}
    T_{-\lambda}^wK_{\lambda_1,\lambda_2}^\nu(g)&=(-1)^{\xi_2}\int_\RR |x|^{\lambda_2-\lambda_1-1}_{\xi_1+\xi_2}|g_{21}+xg_{22}|^{\lambda_1-\nu-\frac{1}{2}}_{\xi_1+\eta}|g_{11}+xg_{12}|^{\nu-\lambda_2-\frac{1}{2}}_{\eta+\xi_2}|\det(g)|^{\lambda_2+1}_{\xi_2}\,dx\\
&=|\det(g)|^{\lambda_2+1}_{\xi_2}(-1)^{\xi_2}\int_\RR |xg_{21}+g_{22}|^{\lambda_1-\nu-\frac{1}{2}}_{\xi_1+\eta}|xg_{11}+g_{12}|^{\nu-\lambda_2-\frac{1}{2}}_{\eta+\xi_2}\,dx,
\end{align*}
where we used the substitution $x\to x^{-1}$ in the second step. This integral can be computed using Corollary~\ref{foldning cor} (for $\lambda$ and $\nu$ in the non-empty open region guaranteeing convergence of the integral) and we find that 
\[
d(\lambda,\nu)=(-1)^{\xi_2+(\xi_1+\eta)(\xi_2+\eta)}\sqrt{\pi}\frac{\Gamma(\frac{\lambda_1-\nu+\frac{1}{2}+[\xi_1+\eta]}{2})\Gamma(\frac{\nu-\lambda_2+\frac{1}{2}+[\xi_2+\eta]}{2})\Gamma(\frac{\lambda_2-\lambda_1+[\xi_1+\xi_2]}{2})}{\Gamma(\frac{\nu-\lambda_1+\frac{1}{2}+[\eta+\xi_1]}{2})\Gamma(\frac{\lambda_2-\nu+\frac{1}{2}+[\xi_2+\eta]}{2})\Gamma(\frac{\lambda_1-\lambda_2+1+[\xi_1+\xi_2]}{2})},
\]
where we have used the notation $[\varepsilon]$ for the unique representative of $\varepsilon\in\ZZ/2\ZZ$ in $\{0,1\}$. We can now use the definition of $\calD(\lambda,\nu)=(\lambda_1-\lambda_2-1)T_{-\lambda_2-1,-\lambda_1}^w\circ M\circ (T_{-w\lambda}^w)^{-1}$ to compute 
\begin{align*}
    (\lambda_1-\lambda_2-1)d(\lambda_2+1,\lambda_1,\nu)K_{\lambda_1,\lambda_2+1}^\nu&=(\lambda_1-\lambda_2-1)T_{-\lambda_2-1,-\lambda_1}^w(\Phi_1K_{\lambda_2,\lambda_1}^\nu)\\
    &=\calD(\lambda,\nu)\circ T_{-\lambda_2,-\lambda_1}^wK_{\lambda_2,\lambda_1}^\nu\\
    &=d(\lambda_2,\lambda_1,\nu)\calD(\lambda,\nu)K_{\lambda_1,\lambda_2}^\nu,
\end{align*}
and cancelling out the various gamma factors we find
\[
\frac{(\lambda_1-\lambda_2-1)d(\lambda_2+1,\lambda_1,\nu)}{d(\lambda_2,\lambda_1,\nu)}=\nu-\lambda_2-\tfrac{1}{2}.
\]
Note that this proves the first identity in Theorem~\ref{thm:IntroA} for the case $n=1$. In total, our Bernstein--Sato identity reads
\[
\calD(\lambda,\nu)K_{\lambda_1,\lambda_2}^\nu=(\nu-\lambda_2-\tfrac{1}{2})K_{\lambda_1,\lambda_2+1}^t \qquad \text{or}\qquad \calD(s,t)K_{s_1,s_2}^t=tK_{s_1,s_2+1}^{t-1}.
\]

\subsection{Analytic continuation and improving normalization}

\noindent Similarly to the construction of $\calD(s,t)$ using multiplication with $\Phi_1$, we can use multiplication with $\Psi_1$ to obtain a differential operator $\calP(s,t)$ depending holomorphically on $s$ and $t$ such that
$$ \calP(s,t)K_{s_1,s_2}^t = s_1K_{s_1-1,s_2+1}^t. $$
With the normalization
$$ \mathbb{K}_{s_1,s_2}^t=\Gamma(s_1+1)^{-1}\Gamma(t+1)^{-1}K_{s_1,s_2}^t, $$
the Bernstein--Sato identities can be rewritten as
\[
\calD(s,t)\mathbb{K}_{s_1,s_2}^t=\mathbb{K}_{s_1,s_2+1}^{t-1}\quad \text{and}\quad \calP(s,t)\mathbb{K}_{s_1,s_2}^t=\mathbb{K}_{s_1-1,s_2+1}^t.
\]
As $\mathbb{K}_{s_1,s_2}^t$ is locally integrable for $\Re(s_1),\Re(t)>-1$ and arbitrary $s_2\in \CC$, we can use these two identities to analytically extend $\mathbb{K}_{s_1,s_2}^t$ to all $(s,t)\in \CC^2\times \CC$, for example from $\{\Re(t)>-1\}$ to $\{\Re(t)>-2\}$ by
\[
    \mathbb{K}_{s_1,s_2}^t := \calD(s,t+1)\mathbb{K}_{s_1,s_2-1}^{t+1}.
\]

However, with this normalization we created some unnecessary zeros of $\mathbb{K}_{s_1,s_2}^t$. This can be seen by rewriting the functional equation \eqref{eq:FunctionalEqForN=1} in terms of the normalized kernels and the normalized Knapp--Stein operators:
\[
\mathbb{K}_{\lambda_1,\lambda_2}^\nu=\frac{\Gamma(\tfrac{\lambda_1-\lambda_2+1+[\xi_1+\xi_2]}{2})\Gamma(\lambda_2-\nu+\tfrac{1}{2})\Gamma(\nu-\lambda_1+\tfrac{1}{2})}{\Gamma(\lambda_1-\nu+\tfrac{1}{2})\Gamma(\nu-\lambda_2+\tfrac{1}{2})}d(\lambda,\nu)\fedT_{-w\lambda}^w\mathbb{K}_{\lambda_2,\lambda_1}^\nu.
\]
Both $\mathbb{K}_{\lambda_1,\lambda_2}^\nu$ and $\fedT_{-w\lambda}^w\mathbb{K}_{\lambda_1,\lambda_2}^\nu$ are holomorphic in $\lambda$ and $\nu$, so that $\mathbb{K}_{\lambda_1,\lambda_2}^\nu$ vanishes at all zeros of the factor in front of $\fedT_{-w\lambda}^w\mathbb{K}_{\lambda_1,\lambda_2}^\nu$. This zero set contains hyperplanes, and multiplying $\mathbb{K}_{\lambda_1,\lambda_2}^\nu$ with gamma factors which have poles on exactly those hyperplanes yields a normalization that still is holomorphic in $\lambda$ and $\nu$:
\[
\fedK_{\lambda_1,\lambda_2}^\nu=\Gamma\big (\tfrac{\lambda_1-\nu+\frac{1}{2}+[\xi_1+\eta]}{2}\big)^{-1}\Gamma\big(\tfrac{\nu-\lambda_2+\frac{1}{2}+[\eta+\xi_2]}{2}\big)^{-1}K_{\lambda_1,\lambda_2}^\nu.
\]
We show in Section~\ref{Section: analytic extension} that this normalization is indeed optimal in the sense that the set of zeros of $\fedK_{\lambda_1,\lambda_2}^\nu$ does not contain any hyperplanes.

\section{Functional identities between Knapp--Stein intertwiners and the distribution kernel}\label{Section funktionalligning}

\noindent One of the key tools in the previous section for $n=1$ was the functional identity \eqref{eq:FunctionalEqForN=1}. In this section we find analogous functional identities for arbitrary $n\geq1$.

The key idea behind the functional identities is that the space $\Hom_H(\pi_{\xi,\lambda}|_H,\tau_{\eta,\nu})$ of symmetry breaking operators between principal series is generically one-dimensional and spanned by $A_{\xi,\lambda}^{\eta,\nu}$. Therefore, the composition of $A_{\xi,\lambda}^{\eta,\nu}$ with a Knapp--Stein intertwining operator for $G$ or for $H$ is proportional to another symmetry breaking operator of this form. For the simple transpositions $w_i$ this implies
$$ A_{\xi,\lambda}^{\eta,\nu} \circ \mathbf{T}_{w_i(\xi,\lambda)}^{w_i} = c_i(\xi,\lambda,\eta,\nu)A_{w_i(\xi,\lambda)}^{\eta,\nu} \qquad \mbox{and} \qquad \mathbf{T}_{\eta,\nu}^{w_i}\circ A_{\xi,\lambda}^{\eta,\nu} = d_i(\xi,\lambda,\eta,\nu)A_{\xi,\lambda}^{w_i(\eta,\nu)} $$
for some constants $c_i(\xi,\lambda,\eta,\nu)$ and $d_i(\xi,\lambda,\eta,\nu)$ that depend meromorphically on $\lambda$ and $\nu$.

To compute these constants, we observe that the adjoint of $T_{w_i(\xi,\lambda)}^{w_i}$ with respect to the $G$-invariant pairing from Section~\ref{sec:PrincipalSeries} is $T_{\xi,-\lambda}^{w_i}$. Since both operators have the same normalization factor, the adjoint of $\mathbf{T}_{w_i(\xi,\lambda)}^{w_i}$ is $\mathbf{T}_{\xi,-\lambda}^{w_i}$. This implies that the distribution kernel of the composition $A_{\xi,\lambda}^{\eta,\nu} \circ \mathbf{T}_{w_i(\xi,\lambda)}^{w_i}f$ equals $\mathbf{T}_{\xi,-\lambda}^{w_i}K_{\xi,\lambda}^{\eta,\nu}$, where we extend $\mathbf{T}_{\xi,-\lambda}^{w_i}$ to the space of distribution vectors of $\pi_{\xi,-\lambda}$ and view $K_{\xi,\lambda}^{\eta,\nu}$ as a distribution vector for $\pi_{\xi,-\lambda}$. Then the functional identity reads
\[
\mathbf{T}_{\xi,-\lambda}^{w_i}K_{\xi,\lambda}^{\eta,\nu}=c_i(\xi,\lambda,\eta,\nu)K_{w_i(\xi,\lambda)}^{\eta,\nu}.
\]
Note that in the case where $K_{\xi,\lambda}^{\eta,\nu}$ is a locally integrable function, the left hand side can be computed using the integral formula \eqref{eq:SimpleTranspositionKSonNbar} for $\mathbf{T}_{\xi,-\lambda}^{w_i}$.

For the composition with Knapp--Stein intertwining operators for $H$, we have to be a bit more careful since intertwining operators for $H$ cannot a priori be applied to distributions on $G$. Nevertheless, in the case where $K_{\xi,\lambda}^{\eta,\nu}$ is a locally integrable function, the distribution kernel of the composition $T_{\eta,\nu}^{w_i}\circ A_{\xi,\lambda}^{\eta,\nu}$ is, in view of \eqref{eq:SBOinTermsOfKernel}, given by applying $T_{\eta,\nu}^{w_i}$ to the function $h\mapsto K_{\xi,\lambda}^{\eta,\nu}(h^{-1}g)$. We therefore define
\begin{equation}
    \mathbf{S}_{\eta,\nu}^{w_i}K_{\xi,\lambda}^{\eta,\nu}(g) := \frac{1}{L(0,\psi_i\psi_{i+1}^{-1})}S_{\eta,\nu}^{w_i}K_{\xi,\lambda}^{\eta,\nu}(g) := \frac{1}{L(0,\psi_i\psi_{i+1}^{-1})}\int_{\RR}K_{\xi,\lambda}^{\eta,\nu}((hw_i\overline{n}_i(x))^{-1}g)\,dx\label{eq:DefinitionKSintertwinerS}
\end{equation}
as the distribution kernel of the composition $T_{\eta,\nu}^{w_i}\circ A_{\xi,\lambda}^{\eta,\nu}$. Then the functional identity reads
$$ \mathbf{S}_{\eta,\nu}^{w_i}K_{\xi,\lambda}^{\eta,\nu}=d_i(\xi,\lambda,\eta,\nu)K_{\xi,\lambda}^{w_i(\eta,\nu)}. $$

\subsection{Identities for minors}\label{sec:IdentitiesMinors}

In order to be able to compute $\mathbf{T}_{\xi,-\lambda}^{w_i}K_{\xi,\lambda}^{\eta,\nu}$ and $\mathbf{S}_{\eta,\nu}^{w_i}K_{\xi,\lambda}^{\eta,\nu}$ we need a few identities involving the minors $\Phi_j(g)$ and $\Psi_j(g)$. We first note that
\begin{align*}
\Phi_j(gw_i\overline{n}_i(x))&=
\begin{cases}
	\Phi_j(g), & j<i,\\
	\Phi_j(gw_i)+x\Phi_j(g), & j=i,\\
	-\Phi_j(g),& j>i,
\end{cases}\\
\Psi_j(gw_i\overline{n}_i(x))&=
\begin{cases}
	\Psi_j(g), & j<i,\\
	\Psi_j(gw_i)+x\Psi_j(g), & j=i,\\
	-\Psi_j(g),& j>i.
\end{cases}
\end{align*}
These identities follow from the fact that the determinant is multilinear and alternating, and similar ones can be obtained for $\Phi_j(\overline{n}_i(x)w_ig)$ and $\Psi_j(\overline{n}_i(x)w_ig)$.
\begin{lemma} \label{algebra}
	For $g\in \GL(n+1,\RR)$ we have the following two identities:
	\begin{align*}
	\Phi_i(g)\Psi_i(gw_i)-\Phi_i(gw_i)\Psi_i(g) &= \Psi_{i-1}(g)\Phi_{i+1}(g),\\
	\Phi_{n+1-i}(w_ig)\Psi_{n-i}(g)-\Phi_{n+1-i}(g)\Psi_{n-i}(w_ig) &= \Psi_{n-i-1}(g)\Phi_{n+2-i}(g),
	\end{align*}
	where we define $\Psi_0=1$.
\end{lemma}
\begin{proof}
	We only show the first identity, the second one follows by analogous arguments. We notice that both sides of the identity have the same equivariance properties from the left by $P_H$. Similarly, both sides have the same equivariance properties from the right by $P_G$, e.g. for $n\in N_G$ we have $w_i n w_i=n'\overline{n}_i(x)$ for some $n'\in N_G$ and $x=n_{i,i+1}$ thus
	\begin{align*}
		\Phi_i(gn)\Psi_i(gnw_i)-\Phi_i(gnw_i)\Psi_i(gn)&=\Phi_i(g)\big[\Psi_i(gw_i)+x\Psi_i(g) \big]-\Psi_i(g)\big[\Phi_i(gw_i)+x\Phi_i(g)\big]\\
		&=\Phi_i(g)\Psi_i(gw_i)-\Phi_i(gw_i)\Psi_i(g).
	\end{align*}
There is an open dense double coset $\omega$ in $P_H\backslash G/ P_G$ which can be described as $\omega=P_HgP_G$ with $g\in G$ satisfying $\Phi_i(g),\Psi_i(g)\neq 0$ for all $i$ (see \cite[Lemma 6.3 \& Lemma 3.6]{F23}). As we are proving an equality of continuous functions, it suffices to check that they coincide on $\omega$. Consider the representative $g$ given by $w_0$ multiplied with the matrix with $1$'s on the diagonal and on the subdiagonal and $0$'s everywhere else. Then the two sides are easily checked to be equal on $g$ and hence on $\omega=P_HgP_G$.
\end{proof}

\subsection{Functional identities for symmetry breaking operators}

As alluded to in the introduction, it simplifies some of the statements if we denote by $\chi_i$ and $\psi_j$ the characters $\chi_i(x)=\vert x \vert^{\lambda_i}_{\xi_i}$ and $\psi_j(x)=\vert x\vert^{\nu_j}_{\eta_j}$, so that
\[
\pi_{\xi,\lambda}=\Ind_{P_G}^G(\chi_1\otimes\cdots\otimes\chi_{n+1}) \qquad \mbox{and} \qquad \tau_{\eta,\nu}=\Ind_{P_H}^H(\psi_1\otimes\cdots\otimes\psi_n).
\]
Recall the definition of the archimedean local $L$-factor $L(s,\chi)$ for a character $\chi$ of $\RR^\times$ from \eqref{eq:DefLocalLFactor}.

\begin{theorem} \label{main theorem}
	\begin{enumerate}
	    \item Let $1\leq i\leq n$ and $(\lambda,\nu)\in\CC^{n+1}\times\CC^n$ such that $\Re(\lambda_i-\nu_{n+1-i})$, $\Re(\nu_{n+1-i}-\lambda_{i+1})>-\frac{1}{2}$ and $\Re(\lambda_i-\lambda_{i+1})<0$. Whenever both sides are defined, the following identity holds:
	   \[
	   \mathbf{T}_{\xi,-\lambda}^{w_i}K_{\xi,\lambda}^{\eta,\nu}=c_i(\xi,\lambda,\eta, \nu)K_{w_i(\xi,\lambda)}^{\eta,\nu},
	   \]
        where
        \[
        c_i(\xi,\lambda,\eta,\nu)=\frac{(-1)^{(\xi_i+\xi_{i+1})(\eta_{n+1-i}+1)+\xi_i\xi_{i+1}}L(\frac{1}{2},\chi_i\psi_{n+1-i}^{-1})L(\frac{1}{2},\chi_{i+1}^{-1}\psi_{n+1-i})}{\sqrt{\pi}L(1,\chi_i\chi_{i+1}^{-1})L(\frac{1}{2},\chi_i^{-1}\psi_{n+1-i})L(\frac{1}{2},\chi_{i+1}\psi_{n+1-i}^{-1})}.
        \]
        \item Let $1\leq i\leq n-1$ and $(\lambda,\nu)\in\CC^{n+1}\times\CC^n$ such that $\Re(\lambda_{n+1-i}-\nu_{i}), \Re(\nu_{i+1}-\lambda_{n+1-i})>-\frac{1}{2}$ and $\Re(\nu_{i+1}-\nu_{i})<0$. Whenever both sides are defined, the following identity holds:
        \[
        \mathbf{S}_{\eta,\nu}^{w_i}K_{\xi,\lambda}^{\eta,\nu}=d_i(\xi,\lambda,\eta,\nu)K_{\xi,\lambda}^{w_i(\eta,\nu)},
        \]
	   where
        \[
        d_i(\xi,\lambda,\eta,\nu)=\frac{(-1)^{(\eta_i+\eta_{i+1})(\xi_{n+1-i}+1)+\eta_i\eta_{i+1}}L(\frac{1}{2},\chi_{n+1-i}\psi_i^{-1})L(\frac{1}{2},\chi_{n+1-i}^{-1}\psi_{i+1})}{\sqrt{\pi} L(1,\psi_{i+1}\psi_i^{-1})L(\frac{1}{2},\chi_{n+1-i}^{-1}\psi_i)L(\frac{1}{2},\chi_{n+1-i}\psi_{i+1}^{-1})}.
        \]
	\end{enumerate}
\end{theorem}

\begin{proof}
To simplify notation we put $|\Psi_{n+1}|_{\varepsilon_{n+1}}^{t_{n+1}}=1 $. For $g$ in the open dense set where $\Phi_i(g),\allowbreak\Phi_i(gw_i),\allowbreak\Psi_i(g),\allowbreak\Psi_i(gw_i)\allowbreak\neq 0$ we have 
	\begin{align*}
		T_{\xi,-\lambda}^{w_i}K_{\xi,\lambda}^{\eta,\nu}(g)
		=\int_\RR K_{\xi,\lambda}^{\eta,\nu}(gw_i\overline{n}_i(x)) dx
		&=(-1)^{\delta_{n+1}+\sum_{j=i+1}^{n}(\delta_j+\varepsilon_j)}\prod_{\substack{j=1 \\ j\neq i}}^{n+1}|\Phi_j(g)|^{s_j}_{\delta_j}|\Psi_j(g)|^{t_j}_{\varepsilon_j}\\
		&\quad \times \int_\RR
		|\Phi_i(gw_i)+x\Phi_i(g)|_{\delta_i}^{s_i}|\Psi_i(gw_i)+x\Psi_i(g)|^{t_i}_{\varepsilon_i} dx.
	\end{align*}
	The integral can be evaluated using Corollary \ref{foldning cor} giving
	\[
	(-1)^{\delta_i}t(s_i,t_i,\delta_i,\varepsilon_i)|\Phi_i(g)|^{-t_i-1}_{\varepsilon_i}|\Psi_i(g)|_{\delta_i}^{-s_i-1}|\Phi_i(g)\Psi_i(gw_i)-\Phi_i(gw_i)\Psi_i(g)|_{\varepsilon_i+\delta_i}^{s_i+t_i+1}.
	\]
	Lastly, applying Lemma~\ref{algebra} we arrive at the first result. 
	The second assertion follows similarly, using \eqref{eq:DefinitionKSintertwinerS} and the second identity in Lemma~\ref{algebra}.
\end{proof}
\begin{remark} \label{alternativ beskrivelse}
Replacing $(\xi,\lambda)$ by $w_i(\xi,\lambda)$, we can phrase the first result in a slightly different way: for $\Re(\lambda_{i+1}-\nu_{n+1-i}),\Re(\nu_{n+1-i}-\lambda_i)>-\frac{1}{2}$ and $\Re(\lambda_{i+1}-\lambda_{i})<0$ we have
	\[
	K_{\xi,\lambda}^{\eta,\nu}=\widetilde{c}_i(\xi,\lambda,\eta,\nu)\mathbf{T}_{w_i(\xi,-\lambda)}^{w_i}K_{w_i(\xi,\lambda)}^{\eta,\nu},
	\]
	where 
 \[
 \widetilde{c}_i(\xi,\lambda,\eta,\nu)=\frac{\sqrt{\pi}L(1,\chi_i^{-1}\chi_{i+1})L(\frac{1}{2},\chi_{i+1}^{-1}\psi_{n+1-i})L(\frac{1}{2},\chi_i\psi_{n+1-i}^{-1})}{(-1)^{(\xi_i+\xi_{i+1})(\eta_{n+1-i}+1)+\xi_i\xi_{i+1}}L(\frac{1}{2},\chi_{i+1}\psi_{n+1-i}^{-1})L(\frac{1}{2},\chi_i^{-1}\psi_{n+1-i})}.
 \]
\end{remark}

In view of \eqref{Knapp-Stein}, we also obtain a functional identity for Knapp--Stein intertwiners associated with Weyl group elements $w$ that are not necessarily simple transpositions. The following special case is used in Section~\ref{Section: analytic extension} to identify zeros of the holomorphic normalization of $K_{\xi,\lambda}^{\eta,\nu}$ obtained from the Bernstein--Sato identities:

\begin{corollary}\label{grimt koro}
	For $1\leq i\leq n$ let $w_+=w_iw_{i+1}\dots w_n$ and $w_-=w_{i-1}w_{i-2}\dots w_1$. Then for $(\lambda,\nu)\in A^{\pm}$ we have
	\[
	K_{\xi,\lambda}^{\eta,\nu}=b_i^\pm(\xi,\lambda,\eta.\nu)\mathbf{T}^{w_\pm}_{w_\pm^{-1}(\xi,-\lambda)}K_{w_\pm^{-1}(\xi,\lambda)}^{\eta,\nu},
	\]
	where
	\begin{equation*}
	b_i^{+}(\xi,\lambda,\eta,\nu)
	=\alpha^+\prod_{j=i}^{n}\frac{L(1,\chi_{j+1}\chi_i^{-1})L(\frac{1}{2},\chi_{j+1}^{-1}\psi_{n+1-j})L(\frac{1}{2},\chi_{i}\psi_{n+1-j}^{-1})}{L(\frac{1}{2},\chi_{j+1}\psi_{n+1-j}^{-1})L(\frac{1}{2},\chi_i^{-1}\psi_{n+1-j})}
	\end{equation*}
	and
	\begin{equation*}
	b_i^{-}(\xi,\lambda,\eta,\nu)\\=\alpha^-\prod_{j=1}^{i-1}\frac{L(1,\chi_i\chi_j^{-1})L(\frac{1}{2},\chi_i^{-1}\psi_{n+1-j})L(\frac{1}{2},\chi_j\psi_{n+1-j}^{-1})}{L(\frac{1}{2},\chi_i\psi_{n+1-j}^{-1})L(\frac{1}{2},\chi_j^{-1}\psi_{n+1-j})},
	\end{equation*}
	$\alpha^{\pm}$ are products of powers of $\pi$ and $(-1)$ that depend on $\xi$ and $\eta$, and
	\[
	A^+=\bigcap_{j=i}^n\Big\{(\lambda,\nu)|\Re(\lambda_{j+1}-\nu_{n+1-j}),\Re(\nu_{n+1-j}-\lambda_i)>-\frac{1}{2},\Re(\lambda_{j+1}-\lambda_i)<0\Big \}\neq \varnothing,
	\]
	and
	\[
	A^-=\bigcap_{j=i}^n\Big\{(\lambda,\nu)|\Re(\lambda_{i}-\nu_{n+1-j}),\Re(\nu_{n+1-j}-\lambda_j)>-\frac{1}{2},\Re(\lambda_{i}-\lambda_j)<0\Big \}\neq \varnothing.
	\]
\end{corollary}
\begin{proof}
	Using \eqref{Knapp-Stein} we find that
	\[
	\mathbf{T}^{w_+}_{w_+^{-1}(\xi,-\lambda)}K_{w_+^{-1}(\xi,\lambda)}^{\eta,\nu}=\mathbf{T}^{w_i}_{w_i(\xi,-\lambda)}\circ\dots\circ \mathbf{T}^{w_n}_{w_n\dots w_i(\xi,-\lambda)}K_{w_n\dots w_i(\xi,\lambda)}^{\eta,\nu}.
	\]
Thus, by successive use of Remark \ref{alternativ beskrivelse} we obtain the claimed formula for $w_+$ with
	\[
	b^+(\xi,\lambda,\eta,\nu)=\widetilde{c}_i(\xi,\lambda,\eta,\nu)\prod_{j=i+1}^{n}\widetilde{c}_j(w_{j-1}\dots w_i(\xi,\lambda),\eta,\nu).
	\]
	The formula for $w_-$ is proven similarly.
\end{proof}

\section{Bernstein--Sato identities for the distribution kernel} \label{section:BS-identities}

\noindent In this section we establish Bernstein--Sato identities for the distribution kernels $K_{\xi,\lambda}^{\eta,\nu}$ for general $n$. The proof follows the same ideas as the one for $n=1$ outlined in Section~\ref{Eksempel}, namely composing multiplication operators with $\Phi_i$ or $\Psi_j$ with Knapp--Stein intertwining operators. Instead of using intertwining operators associated with arbitrary permutations $w$, we only use the ones coming from simple transpositions $w_i$ together with an inductive argument using \eqref{Knapp-Stein}. The major difference between $n=1$ and the general case is that we also need to involve Knapp--Stein intertwining operators for the representations $\tau_{\eta,\nu}$ of $H$ to obtain sufficiently many differential operators to carry out the analytic continuation in Section~\ref{Section: analytic extension}.

\subsection{An \texorpdfstring{$\SL(2)$}{SL(2)}-computation}

The key computation is essentially contained in Section~\ref{eksempel diff operator} for the case $n=1$. Here we give a formulation for arbitrary $n$. It involves certain differential operators that arise from the right regular representation $r$ of $G$ on $C^\infty(G)$ given by $r(g)f(x)=f(xg)$ ($g,x\in G$, $f\in C^\infty(G)$). We consider the differentiated representation $dr$ and for $1\leq i,j\leq n+1$ we put
$$ \varepsilon^{i,j} = dr(E_{i,j}) = \sum_{k=1}^{n+1} g_{k,i}\partial_{g_{k,j}}. $$
For $i=j+1$ the one-parameter group $\exp(xE_{j+1,j})$ is given by $\overline{n}_j(x)$, so that
\[
	\partial_x[f(g\overline{n}_j(x))]=(\varepsilon^{j+1,j}f)(g\overline{n}_j(x))=\varepsilon^{j+1,j}[f(g\overline{n}_j(x))].
\]

\begin{lemma} \label{partiel integration}
	Fix $1\leq i\leq n$. For $\Re(\lambda_i-\lambda_{i+1})\gg1$ and $f\in \pi_{w_i(\xi,-\lambda)}$ we have
	\begin{align*}
		\int_\RR \partial_x\Big(|x|_{\xi_i+\xi_{i+1}}^{\lambda_i-\lambda_{i+1}-1}\Big)f(g\overline{n}_i(x))\,dx&=-\int_\RR|x|_{\xi_i+\xi_{i+1}}^{\lambda_i-\lambda_{i+1}-1} \partial_x [f(g\overline{n}_i(x))]\,dx\\
		&=-\varepsilon^{i+1,i} \int_\RR|x|_{\xi_i+\xi_{i+1}}^{\lambda_i-\lambda_{i+1}-1}f(g\overline{n}_i(x))\,dx.
	\end{align*}
	More generally, if $\calD$ and $\calD'$ are differential operators on $G$ satisfying
    $$ (\calD f)(g\overline{n}_i(x))=\calD' [f(g\overline{n}_i(x))] \qquad \mbox{for all }x\in\RR, $$
    then 
	\[
	\int_\RR|x|_{\xi_i+\xi_{i+1}}^{\lambda_i-\lambda_{i+1}-1}  [\calD f](g\overline{n}_i(x))\,dx=\calD' \int_\RR|x|_{\xi_i+\xi_{i+1}}^{\lambda_i-\lambda_{i+1}-1}f(g\overline{n}_i(x))\,dx,
	\]
	where on the right hand sides $\mathcal{D}'$ is differentiating in $g$.
\end{lemma}
\begin{proof}
	For the first identity we only have to argue that there are no boundary terms when integrating by parts. Decomposing $\overline{n}_i(x)=k(x)a(x)n(x)\in KAN$ according to \eqref{iwasawa} yields
	\begin{align*}
		|x|_{\xi_i+\xi_{i+1}}^{\lambda_i-\lambda_{i+1}-1}f(g\overline{n}_i(x))&=|x|_{\xi_i+\xi_{i+1}}^{\lambda_i-\lambda_{i+1}-1}(1+x^2)^{\frac{\lambda_{i+1}-\lambda_{i}-1}{2}}f(gk(x))\\
		&=|x|_{\xi_i+\xi_{i+1}}^{-1}(1+x^2)^{-\frac{1}{2}}\Big(\frac{|x|}{\sqrt{1+x^2}}\Big)^{\lambda_i-\lambda_{i+1}}f(gk(x)).
	\end{align*}
The right hand side clearly vanishes at $x=\pm\infty$ giving us no boundary term. It remains to show the last statement. For $g$ in a compact subset of $G$ we get, using the Iwasawa decomposition, that
	\begin{align*}
		\Big||x|_{\xi_i+\xi_{i+1}}^{\lambda_i-\lambda_{i+1}-1}\calD'[f(g\overline{n}_i(x))]\Big|&=\Big ||x|_{\xi_i+\xi_{i+1}}^{\lambda_i-\lambda_{i+1}-1}(1+x^2)^{\frac{\lambda_{i+1}-\lambda_{i}-1}{2}}\calD'[f(gk(x))]\Big|\\
		&\leq c(1+x^2)^{\frac{\Re(\lambda_{i+1}-\lambda_{i}-1)}{2}}|x|^{\Re(\lambda_i-\lambda_{i+1}-1)},
	\end{align*}
where $c$ is a constant depending on $f$, $g$ and $\mathcal{D}'$. The right hand side is integrable, allowing us to pull the differential operator out of the integral.
\end{proof}

\subsection{Bernstein--Sato identities}

In addition to the differential operators $\varepsilon^{i,j}$ arising from the right regular representation, we consider the left regular representation $\ell$ of $G$ given by $\ell(g)f(x)=f(g^{-1}x)$ ($g,x\in G$) and the corresponding operators given in terms of the differentiated representation $d\ell$:
$$ \varepsilon_{i,j}:=-d\ell(E_{j,i}) = \sum_{k=1}^{n+1}g_{i,k}\partial_{g_{j,k}}  \qquad (1\leq i,j\leq n+1). $$
We further abbreviate
$$ \lambda_{i,j}:=\lambda_i-\lambda_j-1 \qquad \mbox{and} \qquad \nu_{i,j}:=\nu_i-\nu_j-1 \qquad (1\leq i,j\leq n+1). $$
For $1\leq i,j\leq n+1$ we define the Weyl group element $w_{i,j}$ by
\[
w_{i,j}=\begin{cases}
	w_iw_{i+1}\cdots w_{j}, &\text{for } i< j, \\
	w_iw_{i-1}\cdots w_j, & \text{for }  i> j,\\
	w_i, & \text{for } i=j.
\end{cases}
\]

Now, consider the differential operators
\[
\mathcal{D}_i(\lambda)=(-1)^{i+1}
\det\left( \begin{matrix}
	\Phi_1 & \lambda_{1,i+1} & 0 & \cdots & 0\\
	r(w_{1,1})\Phi_1 &\varepsilon^{2,1}& \lambda_{2,i+1}& \cdots &0\\
	\vdots & \vdots & \vdots & \ddots & \vdots\\
	r(w_{i-1},1)\Phi_1 & \varepsilon^{i,1} & \varepsilon^{i,2}&\cdots  & \lambda_{i,i+1} \\
	r(w_{i,1})\Phi_1 & \varepsilon^{i+1,1} & \varepsilon^{i+1,2}&\cdots & \varepsilon^{i+1,i} 
\end{matrix}\right)\quad (1\leq i\leq n),
\]
\[
\mathcal{C}_i(\nu)=(-1)^{n-i}
\det\left( \begin{matrix}
	\Psi_1 & \nu_{n,i} & 0 & \cdots & 0\\
	\ell(w_{n-1,n-1})\Psi_1 &\varepsilon_{n-1,n}& \nu_{n-1,i}& \cdots &0\\
	\vdots & \vdots & \vdots & \ddots & \vdots\\
	\ell(w_{i+1,n-1})\Psi_1& \varepsilon_{i+1,n} & \varepsilon_{i+1,n-1}&\cdots  & \nu_{i+1,i} \\
	\ell(w_{i,n-1})\Psi_1& \varepsilon_{i,n} & \varepsilon_{i,n-1}&\cdots & \varepsilon_{i,i+1} 
\end{matrix}\right)\quad (1\leq i\leq n-1).
\]
Moreover, we write $r(w)\calC_i(\nu)$ for the same determinant as for $\calC_i(\nu)$ but where $r(w)$ is applied to all the $\Psi_1$ in the first column, and to ease notation we also put $\calC_n(\nu)=\Psi_1$. Using this notation, we additionally consider the following differential operators:
\[
\mathcal{L}_i(\lambda,\nu)=(-1)^{i+1}
\det\left( \begin{matrix}
	\calC_{n-i}(\nu) & \lambda_{1,i+1} & 0 & \cdots & 0\\
	r(w_{1,1})\calC_{n-i}(\nu) &\varepsilon^{2,1}& \lambda_{2,i+1}& \cdots &0\\
	\vdots & \vdots & \vdots & \ddots & \vdots\\
	r(w_{i-1},1)\calC_{n-i}(\nu) & \varepsilon^{i,1} & \varepsilon^{i,2}&\cdots  & \lambda_{i,i+1} \\
	r(w_{i,1})\calC_{n-i}(\nu) & \varepsilon^{i+1,1} & \varepsilon^{i+1,2}&\cdots & \varepsilon^{i+1,i} 
\end{matrix}\right)\quad (1\leq i\leq n-1),
\]
\[
\mathcal{P}_i(\lambda,\nu)=(-1)^{i+1}
\det\left( \begin{matrix}
	\calC_{n+1-i}(\nu) & \lambda_{1,i+1} & 0 & \cdots & 0\\
	r(w_{1,1})\calC_{n+1-i}(\nu) &\varepsilon^{2,1}& \lambda_{2,i+1}& \cdots &0\\
	\vdots & \vdots & \vdots & \ddots & \vdots\\
	r(w_{i-1},1)\calC_{n+1-i}(\nu) & \varepsilon^{i,1} & \varepsilon^{i,2}&\cdots  & \lambda_{i,i+1} \\
	r(w_{i,1})\calC_{n+1-i}(\nu) & \varepsilon^{i+1,1} & \varepsilon^{i+1,2}&\cdots & \varepsilon^{i+1,i} 
\end{matrix}\right)\quad (1\leq i\leq n).
\]
As the entries of these determinants are non-commuting, we specify that this determinant should be considered as
\[
\det A=\sum_{\sigma\in \mathbb{S}_n} \sgn(\sigma)a_{\sigma(1),1}a_{\sigma(2),2}\cdots a_{\sigma(n),n}.
\]
Note that $\calD_i(\lambda)$ has order $i$, $\calC_i(\nu)$ has order $n-i$, $\calL_i(\lambda,\nu)$ has order $2i$ and $\calP_i(\lambda,\nu)$ has order $2i-1$.

We now show that all these differential operators arise from multiplication operators composed from left and right by appropriate Knapp--Stein intertwining operators. For a $\CC$-valued function $\Phi$ let $M_\Phi$ be the multiplication operator given by $M_\Phi f=\Phi f$.

\begin{proposition} \label{diff op}
For generic $(\lambda,\nu)\in\CC^{n+1}\times\CC^n$ such that $\Re(\lambda_1)\gg\ldots\gg \Re(\lambda_{n+1})$ and $\Re(\nu_n)\gg  \ldots \gg \Re(\nu_1)$ the differential operators can be expressed in the following way:
\begin{align*}
		\calD_1(\lambda)&=\lambda_{1,2}T_{w_1(\xi,-\lambda)-(e_1,e_1)}^{w_1}\circ M_{\Phi_1}\circ \Big(T^{w_1}_{w_1(\xi,-\lambda)}\Big)^{-1},\\
		\calD_i(\lambda)&=\lambda_{i,i+1}T_{w_i(\xi,-\lambda)-(e_i,e_i)}^{w_i}\circ \calD_{i-1}(w_i\lambda)\circ \Big(T^{w_i}_{w_i(\xi,-\lambda)}\Big)^{-1} && (2\leq i\leq n),\\
		\calC_{n-1}(\nu)&=\nu_{n,n-1}S_{w_{n-1}(\eta,\nu)+(e_n,e_n)}^{w_{n-1}}\circ M_{\Psi_1}\circ \Big(S^{w_{n-1}}_{w_{n-1}(\eta,\nu)}\Big)^{-1},\\
		\calC_i(\nu)&=\nu_{i+1,i}S_{w_i(\eta,\nu)+(e_{i+1},e_{i+1})}^{w_i}\circ \calC_{i+1}(w_i\nu)\circ \Big(S^{w_i}_{w_i(\eta,\nu)}\Big)^{-1} && (1\leq i\leq n-2),\\
  \calP_1(\lambda)&=\lambda_{1,2}T^{w_1}_{-w_1(\xi,-\lambda)-(e_1,e_1)}\circ M_{\Psi_1}\circ\Big (T^{w_1}_{w_1(\xi,-\lambda)}\Big)^{-1},\\
  \calL_i(\lambda,\nu)&=\nu_{n+1-i,n-i}S^{w_{n-i}}_{w_{n-i}(\eta,\nu)+(e_{n+1-i},e_{n+1-i})}\circ \calP_{i}(\lambda,w_{n-i}\nu)\circ \Big(S^{w_{n-i}}_{w_{n-i}(\eta,\nu)}\Big)^{-1}, && (1\leq i \leq n-1),\\
  \calP_i(\lambda,\nu)&=\lambda_{i,i+1}T^{w_i}_{w_i(\xi,-\lambda)-(e_i,e_i)}\circ \calL_{i-1}(w_i\lambda,\nu)\circ \Big(T^{w_i}_{w_i(\xi,-\lambda)}\Big)^{-1}, && (2\leq i\leq n).
	\end{align*}
\end{proposition}

Note that for generic $\lambda$ the intertwining operator $T^w_{w(\xi,-\lambda)}$ is invertible, so $(T^w_{w(\xi,-\lambda)})^{-1}$ exists. The same is true for $S^w_{w(\eta,\nu)}$.

\begin{proof}
We start by showing the formulas for the family $\calD_i(\lambda)$. For $i=1$ and $f\in \pi_{w_1(\xi,-\lambda)}$, we imitate the argument in Section~\ref{eksempel diff operator}, using \eqref{eq:SimpleTranspositionKSonNbar}, the identity $\Phi_1(g\overline{n}_1(x))=\Phi_1(g)+x(r(w_1)\Phi_1)(g)$ (see Section~\ref{sec:IdentitiesMinors}) and Lemma~\ref{partiel integration}, to find
\begin{equation*}
    \lambda_{1,2}T^{w_1}_{w_1(\xi,-\lambda)-(e_1,e_1)}(\Phi_1f)(g)=-[\Phi_1(g)\varepsilon^{2,1}-r(w_1)\Phi_1(g)\lambda_{1,2}]T^{w_1}_{w_1(\xi,-\lambda)}f(g).
\end{equation*}
This is the claimed formula for $\calD_1(\lambda)=-[\Phi_1(g)\varepsilon^{2,1}-\lambda_{1,2}r(w_1)\Phi_1(g)]$. Next, we show the formula for $\calD_i(\lambda)$ ($2\leq i\leq n$), i.e. for $f\in \pi_{w_i(\xi,-\lambda)}$ we show
$$ \lambda_{i,i+1}T_{w_i(\xi,-\lambda)-(e_i,e_i)}^{w_i}\circ\calD_{i-1}(w_i\lambda)f = \calD_i(\lambda)\circ T_{w_i(\xi,-\lambda)}^{w_i}f. $$
In view of \eqref{eq:SimpleTranspositionKSonNbar}, the left hand side evaluated at $g\in G$ is given by
$$ (-1)^{\xi_i}\lambda_{i,i+1}\int_\RR |x|_{\xi_i+\xi_{i+1}+1}^{\lambda_i-\lambda_{i+1}-2}(\calD_{i-1}(w_i\lambda)f)(g\overline{n}_i(x))\,dx. $$
Consider the term $(\calD_{i-1}(w_i\lambda)f)(g\overline{n}_i(x))$ with $\calD_{i-1}(w_i\lambda)$ given by the determinant formula from the beginning of this section. Using the formulas in Section~\ref{sec:IdentitiesMinors}, we find that for $j\leq i-1$ we have $[r(w_{j,1})\Phi_1](g\overline{n}_{i}(x))=[r(w_{j,1})\Phi_1](g)+\delta_{i-1,j}x[r(w_i)r(w_{j,1})\Phi_1](g)$. Moreover, for $k<j\leq i$ we have by the chain rule $(\varepsilon^{j,k}f)(g\overline{n}_{i}(x))=(\varepsilon^{j,k}+\delta_{i,j}x \varepsilon^{j+1,k} )[f(g\overline{n}_{i}(x))]$. Thus, only the last row of the matrix defining $\calD_{i-1}(w_i\lambda)$ does not commute with the right translation by $\overline{n}_i(x)$, and by linearity in the last row of this determinant we get
	\[
	(\calD_{i-1}(w_{i}\lambda)f)(g\overline{n}_{i}(x))=\calD_{i-1}(w_{i}\lambda) [f(g\overline{n}_{i}(x))]+x\widetilde{\calD}_{i-1}(w_{i}\lambda) [f(g\overline{n}_{i}(x))],
	\]
where $\widetilde{\calD}_{i-1}(w_{i}\lambda)$ is the determinant of the same matrix as the one for $\calD_{i-1}(w_{i}\lambda)$, but where $r(w_{1,i-1})$ is replaced by $r(w_{1,i})$ and $\varepsilon^{i,k}$ is replaced by $\varepsilon^{i+1,k}$ for $1\leq k\leq i-1$. Now, using the same steps as for $i=1$, but with $\calD_{i-1}(w_i\lambda)+x\widetilde{\calD}_{i-1}(w_i\lambda)$ playing the role of $\Phi_1(g)+xr(w_1)\Phi_1(g)$, and using Lemma~\ref{partiel integration}, we get
	\[
	\lambda_{i,i+1}T_{w_{i}(\xi,-\lambda)-(e_{i},e_{i})}^{w_i}(\calD_{i-1}(w_{i}\lambda)f)(g)
	=-\left[\calD_{i-1}(w_{i}\lambda)\varepsilon^{i+1,i}-\widetilde{\calD}_{i-1}(w_{i}\lambda)\lambda_{i,i+1} \right]T_{w_{i}(\xi,-\lambda)}^{w_i}f(g).
	\]
Computing the determinant defining $\calD_i(w_i\lambda)$ along the last column shows that
$$ -\calD_{i-1}(w_{i}\lambda)\varepsilon^{i+1,i}+\widetilde{\calD}_{i-1}(w_{i}\lambda)\lambda_{i,i+1}=\calD_i(\lambda). $$
The proof for $\calC_i(\nu)$ is identical, using $S_{\eta,\nu}^w$ instead of $T_{\xi,\lambda}^w$.

To prove the identities for $\calL_i(\lambda,\nu)$ and $\calP_i(\lambda,\nu)$, we observe that the operators $T^{w_i}$ and $S^{w_j}$ commute. Therefore, we can first apply all operators $S^{w_j}$, which results in respectively $\calC_{n+1-i}(\nu)$ and $\calC_{n-i}(\nu)$. These have to be conjugated by all the operators $T^{w_i}$. For this, we note that $(\calC_i(\nu)f)(g\overline{n}_j(x))=[\calC_i(\nu)+\delta_{1,j}xr(w_1)\calC_i(\nu)]f(g\overline{n}_j(x))$ since $\varepsilon_{i,j}$ is the derivative of the left regular action and thus commutes with the right regular action, and since $\Psi_1$ satisfies the same identity by multi-linearity of the determinant. Obtaining the claimed expressions for $\calP_i(\lambda,\nu)$ and $\calL_i(\lambda,\nu)$ then follows in exactly the same manner as for $\calD_i(\lambda)$, but where $\Phi_1$ is replaced by the corresponding $\calC_i(\nu)$.
\end{proof}

\begin{remark}
Most of the differential operators defined above are not $H$-intertwining since the multiplication maps $M_{\Phi_i}$ and $M_{\Psi_j}$ are in general not $H$-intertwining. However, since $\Phi_1$ and $\Psi_n$ are $H$-equivariant, the multiplication operators $M_{\Phi_1}$ and $M_{\Psi_n}$ are $H$-intertwining and therefore this procedure generates differential $H$-intertwining operators. For the polynomial $\Phi_1$ we obtain the differential operators $\calD_i(\lambda)$, and using $\Psi_n$ we obtain another family of differential operators $\calF_i(\lambda)$ defined by
\begin{align*}
    \calF_n(\lambda)&=\lambda_{n,n+1}T_{w_n(\xi,-\lambda)-(\mathds{1}-e_{n+1},\mathds{1}-e_{n+1})}^{w_n}\circ M_{\Psi_n}\circ \Big(T^{w_{n}}_{w_n(\xi,-\lambda)}\Big)^{-1},\\
		\calF_i(\lambda)&=\lambda_{i,i+1}T_{w_i(\xi,-\lambda)-(\mathds{1}-e_{i+1},\mathds{1}-e_{i+1})}^{w_i}\circ \calF_{i+1}(w_i\lambda)\circ \Big(T^{w_i}_{w_i(\xi,-\lambda)}\Big)^{-1} && (1\leq i\leq n-1),\\
\end{align*}
and given explicitly as 
\[
\mathcal{F}_{i}(\lambda)=(-1)^{n+1-i}
\det\left( \begin{matrix}
	\Psi_n & \lambda_{i,n+1} & 0 & \cdots & 0\\
	r(w_{n,n})\Psi_n &\widetilde{\varepsilon}^{n+1,n}& \lambda_{i,n}& \cdots &0\\
	\vdots & \vdots & \vdots & \ddots & \vdots\\
	r(w_{i+1,n})\Psi_n & \widetilde{\varepsilon}^{n+1,i+1} & \widetilde{\varepsilon}^{n,i+1}&\cdots  & \lambda_{i,i+1} \\
	r(w_{i,n})\Psi_n& \widetilde{\varepsilon}^{n+1,i} & \widetilde{\varepsilon}^{n,i}&\cdots & \widetilde{\varepsilon}^{i+1,i} 
\end{matrix}\right)\quad (1\leq i\leq n),
\]
where $\widetilde{\varepsilon}^{i,j}=(-1)^{i+j+1}\varepsilon^{i,j}$. The proof for the expression follows in the exact same manner as for $\calD_i$ but with a little twist. As for $k\geq i$ we have $[r(w_{k,n})\Psi_n](g\overline{n}_{i-1}(x))=[r(w_{k,n})\Psi_n](g)+\delta_{i,j}x[r(w_{k-1}r(w_{k,n})\Psi_n](g) $ and $ (\varepsilon^{k+1,i}f)(g\overline{n}_{i-1}(x))\allowbreak=[\varepsilon^{k+1,i}-x\varepsilon^{k+1,i-1}] \allowbreak f(g\overline{n}_{i-1}(x))$, we find opposite signs in the last row when trying to use multilinearity as above. However,  changing to the $\sim$-notation we can rewrite the latter formula as $(\widetilde{\varepsilon}^{k+1,i}f)(g\overline{n}_{i-1}(x))\allowbreak=[\widetilde{\varepsilon}^{k+1,i}+x\widetilde{\varepsilon}^{k+1,i-1}] f(g\overline{n}_{i-1}(x))$ and the proof follows in the same manner as for $\calD_i(\lambda)$.

Composing $\calD_i$ and $\calF_i$ with the restriction map, restricting functions from $G$ to $H$, produces differential symmetry breaking operators between principal series of $G$ and $H$. In a subsequent paper, these will be studied in more detail and in particular related to residues of the family $A_{\xi,\lambda}^{\eta,\nu}$ of symmetry breaking operators.

\end{remark}
\begin{theorem}[Bernstein--Sato identities] \label{Bernstein Sato}
	The following Bernstein--Sato identities hold:
	\begin{align*}
		\mathcal{D}_i(\lambda)K_{\xi,\lambda}^{\eta,\nu}&=p_{\mathcal{D}_i}(\lambda,\nu)K_{(\xi+e_{i+1},\lambda+e_{i+1})}^{\eta,\nu},\\
  \calP_i(\lambda,\nu)K_{\xi,\lambda}^{\eta,\nu}&=p_{\calP_i}(\lambda,\nu)K_{(\xi,\lambda)+(e_{i+1},e_{i+1})}^{(\eta,\nu)+(e_{n+1-i},e_{n+1-i})},\\
    \calL_i(\lambda,\nu)K_{\xi,\lambda}^{\eta,\nu}&=p_{\calL_i}(\lambda,\nu)K_{(\xi,\lambda)+(e_{i+1},e_{i+1})}^{(\eta,\nu)+(e_{n-i},e_{n-i})},
	\end{align*}
	where
	\begin{align*}
		p_{\calD_i}(\lambda,\nu)&=(-1)^{i+1}\prod_{j=n+1-i}^n (\nu_{j}-\lambda_{i+1}-\tfrac{1}{2} ),\\
  p_{\calP_i}(\lambda,\nu)&= \prod_{j=1}^{i}(\lambda_j-\nu_{n+1-i}-\tfrac{1}{2})\prod_{k=n+2-i}^n(\nu_k-\lambda_{i+1}-\tfrac{1}{2}),\\
  p_{\calL_i}(\lambda,\nu) &= \prod_{k=1}^i(\lambda_k-\nu_{n-i}-\tfrac{1}{2})\prod_{j=n+1-i}^n(\nu_j-\lambda_{i+1}-\tfrac{1}{2}).
	\end{align*}
\end{theorem}
\begin{proof}
Using Theorem \ref{main theorem} and the same argument as in Subsection \ref{eksempel funktionalligning} but for all simple transpositions $w_i$ we get the result.
\end{proof}

\begin{remark}
	The differential operators $\calD_i(\lambda)$, $\calC_i(\nu)$, $\calF_i(\lambda)$, $\calP_i(\lambda,\nu)$, $\calL_i(\lambda,\nu)$ are defined only in terms of the multiplication operators by $\Phi_1$, $\Psi_1$ and $\Psi_n$. This choice might seem arbitrary, but it is sufficient for our purpose: as explained in the next section the obtained Bernstein--Sato identities extend $K_{\xi,\lambda}^{\eta,\nu}$ meromorphically to all $(\lambda,\nu)\in\CC^{n+1}\times\CC^n$. However, many different choices of multiplication operators by $\Phi_i$ and $\Psi_i$ and permutations $w$ also give collections of differential operators that can extend $K_{\xi,\lambda}^{\eta,\nu}$, but there are a few immediate limitations:\\
	For a choice of $\Phi_i$ or $\Psi_i$, say $\Phi_i$, there is only one choice of simple transposition for which the composition of $M_{\Phi_i}$ with $T^{w_j}$ yields a non-trivial differential operator since $T^{w_j}$ and $M_{\Phi_i}$ commute when $j\neq i$. Moreover, if we let $n=2$ and consider every outcome of conjugating $\Phi_i$ or $\Psi_i$ with only intertwiners $T^{w_j}$ for $\pi_{\xi,-\lambda}$, then we never obtain sufficiently many differential operators in order to extend $K_{\xi,\lambda}^{\eta,\nu}$ to all $(\lambda,\nu)\in\CC^{n+1}\times\CC^n$. This suggests that at some point the intertwiners $S^{w_j}$ for $\tau_{\eta,\nu}$ must be used.
\end{remark}

\section{Analytic extension of the integral kernel}\label{Section: analytic extension}
\noindent In this section we use the Bernstein--Sato identities and the functional identities to find the appropriate normalization factors that allow to analytically extend the distribution kernel $K_{\xi,\lambda}^{\eta,\nu}$. This is done in two steps. First, we normalize the kernel so that the Bernstein--Sato identities provide an analytic continuation. In the second step we use the functional identities to identify a large set of zeros of the analytic extension and hence obtain a renormalization that still is holomorphic, but has less zeros.

\subsection{Holomorphic normalization}

Consider the normalized kernel
\[
\mathbb{K}_{\xi,\lambda}^{\eta,\nu}=\frac{K_{\xi,\lambda}^{\eta,\nu}}{\prod_{i+j\leq n+1}\Gamma(\lambda_i-\nu_j+\tfrac{1}{2})\times\prod_{i+j>n+1}\Gamma(\nu_j-\lambda_i+\tfrac{1}{2})},
\]
i.e. we normalize the kernel with gamma factors corresponding to all linear factors of the Bernstein--Sato polynomials from Theorem \ref{Bernstein Sato}. Rewriting the Bernstein--Sato identities of Theorem \ref{Bernstein Sato} in terms of $\mathbb{K}_{\xi,\lambda}^{\eta,\nu}$ gives
\begin{equation}\label{eq:BSidentitiesForBBK}
\begin{split}
	\mathcal{D}_i(\lambda)\mathbb{K}_{\xi,\lambda}^{\eta,\nu} &= (-1)^{i+1}\prod_{j=1}^{n-i}(\lambda_{i+1}-\nu_j+\tfrac{1}{2})\cdot\mathbb{K}_{\xi+e_{i+1},\lambda+e_{i+1}}^{\eta,\nu},\\
    \calP_i(\lambda,\nu)\mathbb{K}_{\xi,\lambda}^{\eta,\nu} &= \prod_{j=1}^{n-i}(\lambda_{i+1}-\nu_{j}+\tfrac{1}{2})\prod_{k=i+2}^{n+1}(\nu_{n+1-i}-\lambda_k+\tfrac{1}{2})\cdot \mathbb{K}_{\xi+e_{i+1},\lambda+e_{i+1}}^{\eta+e_{n+1-i},\nu+e_{n+1-i}},\\
    \calL_i(\lambda,\nu)\mathbb{K}_{\xi,\lambda}^{\eta,\nu}&= \prod_{j=1}^{n-1-i}(\lambda_i-\nu_j+\tfrac{1}{2})\prod_{k=i+2}^{n+1}(\nu_{n-i}-\lambda_k+\tfrac{1}{2})\cdot \mathbb{K}_{\xi+e_{i+1},\lambda+e_{i+1}}^{\eta+e_{n-i},\nu+e_{n-i}}.
\end{split}
\end{equation}

\begin{proposition}
$\mathbb{K}_{\xi,\lambda}^{\eta,\nu}$ extends analytically to a family of distributions in $\calD'(G)$ that depends holomorphically on $(\lambda,\nu)\in \CC^{n+1}\times \CC^n$. Moreover, $\mathbb{K}_{\xi,\lambda}^{\eta,\nu}\in \calD'(G)_{\xi,\lambda}^{\eta,\nu}$ for all $\xi$, $\eta$, $\lambda$ and $\nu$.
\end{proposition}

\begin{proof}
	Abusing notation, we rewrite the Bernstein--Sato identities for $\calD_n$, $\calP_i$ and $\calL_i$ in terms of the spectral parameters $\delta$, $\varepsilon$, $s$ and $t$:

 \begin{align*}
     \calD_n(s,t)\mathbb{K}_{\delta,s}^{\varepsilon,t} &= (-1)^{n+1}\mathbb{K}_{\delta+e_{n+1},s+e_{n+1}}^{\varepsilon-e_n,t-e_n},\\
     \calP_i(s,t)\mathbb{K}_{\delta,s}^{\varepsilon,t} &= \prod_{j=1}^{n-i}(s_{i+1}+t_{i+1}+\cdots+s_{n+1-j}+n+1-i-j)\\
     &\quad \times \prod_{k=i+2}^{n+1}(t_i+s_{i+1}+\cdots+t_{k-1}+k-i)\cdot\mathbb{K}_{\delta+e_{i+1}-e_i,s+e_{i+1}-e_i}^{\varepsilon,t},\\
     \calL_i(s,t)\mathbb{K}_{\delta,s}^{\eta,\nu} &= \prod_{j=1}^{n-i-1}(s_{i+1}+t_{i+1}+\cdots+s_{n+1-j}+n+1-i-j)\\
     &\quad \times \prod_{k=i+2}^{n+1}(t_{i+1}+s_{i+2}+\cdots+t_{k-1}+k-1-i)\cdot\mathbb{K}_{\delta,s}^{\varepsilon+e_{i+1}-e_i,t+e_{i+1}-e_i}.\\
 \end{align*}
The family $\mathbb{K}_{\delta,s}^{\varepsilon,t}$ is locally integrable and holomorphic in $s$ and $t$ in the region
    $$ \{(s,t):\Re(s_1),\ldots,\Re(s_n),\Re(t_1),\ldots,\Re(t_n)>0\} $$
and hence defines a holomorphic family of distributions in $\calD'(G)$ there. We extend this family first to all $t_n\in\CC$, then to all $s_n\in\CC$ then to all $t_{n-1}\in \CC$ and so on. First observe that the right hand side of the Bernstein--Sato identity for $\calD_n$ does not contain any linear factors in $s$ and $t$, so using $\calD_n(s,t)$ we can extend $\mathbb{K}_{\delta,s}^{\varepsilon,t}$ to all $t_n\in \CC$. The analogous argument using $\calP_n$ extends to all $s_n\in \CC$. Next, we would like to use $\calL_{n-1}(s,t)$ to extend to all $t_{n-1}\in \CC$. The corresponding Bernstein--Sato identity reads
\[
\calL_{n-1}(s,t)\mathbb{K}_{\delta,s}^{\varepsilon,t}=(t_n+1)\cdot \mathbb{K}_{\delta,s}^{\varepsilon+e_{n}-e_{n-1},t+e_n-e_{n-1}},
\]
so the Bernstein--Sato polynomial on the right hand side is independent of $t_{n-1}$. Therefore, its set of zeros is also independent of $t_{n-1}$. This implies that whenever the Bernstein--Sato polynomial is zero we have $\calL_{n-1}(s,t)\mathbb{K}_{\delta,s}^{\varepsilon,t}=0$ for all $t_{n-1}\in \CC$ for which $\mathbb{K}_{\delta,s}^{\varepsilon,t}$ is defined and holomorphic. Hence, we can divide $\calL_{n-1}(s,t)\mathbb{K}_{\delta,s}^{\varepsilon,t}$ by the Bernstein--Sato polynomial and obtain a holomorphic function in $t_{n-1}$ which we use to define $\mathbb{K}_{\delta,s}^{\varepsilon+e_n-e_{n-1},t+e_n-e_{n-1}}$. In this way we can extend $\mathbb{K}_{\delta,s}^{\varepsilon,t}$ holomorphically to all $t_{n-1}\in \CC$.\\
By the same argument, switching between $\calP_i(s,t)$ and $\calL_i(s,t)$, we can extend to all parameters $s_1,\dots,s_{n+1},t_1,\dots,t_n\in \CC$. Here we use that the Bernstein--Sato polynomial to extend to the next variable only depends on the variables to which we already have extended in the previous steps. This extends the kernel to all spectral parameters and therefore to all principal series parameters.
\end{proof}

\subsection{Improved renormalization}

Using the duplication formula for the gamma function we can write
\[
\Gamma(\lambda_i-\nu_j+\tfrac{1}{2})=\frac{2^{\lambda_i-\nu_j-\frac{1}{2}}}{\sqrt{\pi}}\Gamma\Big(\frac{\lambda_i-\nu_j+\frac{1}{2}+[\xi_i+\eta_j]}{2}\Big)\Gamma\Big(\frac{\lambda_i-\nu_j+\frac{3}{2}-[\xi_i+\eta_j]}{2}\Big).
\]
Recall that $[\varepsilon]$ denotes the unique representative of $\varepsilon\in\ZZ/2\ZZ$ in $\{0,1\}$. Comparing this with the analytically extended Remark \ref{alternativ beskrivelse} we get
\[
\mathbb{K}_{\xi,\lambda}^{\eta,\nu}=\beta_i(\xi,\lambda,\eta,\nu)\mathbf{T}_{w_i(\xi,-\lambda)}^{w_i}\mathbb{K}_{w_i(\xi,\lambda)}^{\eta,\nu},
\]
where
\[
\beta_i(\xi,\lambda,\eta,\nu)=\frac{\Gamma(\frac{\lambda_{i+1}-\lambda_{i}+1+[\xi_i+\xi_{i+1}]}{2})\Gamma(\frac{\lambda_{i+1}-\nu_{n+1-i}+\frac{3}{2}-[\xi_{i+1}+\eta_{n+1-i}]}{2})\Gamma(\frac{\nu_{n+1-i}-\lambda_{i}+\frac{3}{2}-[\xi_{i}+\eta_{n+1-i}]}{2})}{c\Gamma(\frac{\nu_{n+1-i}-\lambda_{i+1}+\frac{3}{2}-[\xi_{i+1}+\eta_{n+1-i}]}{2})\Gamma(\frac{\lambda_{i}-\nu_{n+1-i}+\frac{3}{2}-[\xi_{i}+\eta_{n+1-i}]}{2})}
\]
and $c$ is a product of powers of $\pi$, $2$ and $-1$. As $\mathbb{K}_{\xi,\lambda}^{\eta,\nu}$ and $\mathbf{T}_{w_i(\xi,-\lambda)}^{w_i}\mathbb{K}_{w_i(\xi,\lambda)}^{\eta,\nu}$ are holomorphic in $\lambda$ and $\nu$, we conclude that $\mathbb{K}_{\xi,\lambda}^{\eta,\nu}$ vanishes at all zeros of $\beta_i(\xi,\lambda,\eta,\nu)$. This shows that $\mathbb{K}_{\xi,\lambda}^{\eta,\nu}$ is over-normalized and the factors $\Gamma(\lambda_i-\nu_{n+1-i}+\frac{1}{2})$, $\Gamma(\nu_{n+1-i}-\lambda_{i+1}+\frac{1}{2})$ can be replaced by $\Gamma(\frac{\lambda_i-\nu_{n+1-i}+\frac{1}{2}+[\xi_i+\eta_{n+1-i}]}{2})$,  $\Gamma(\frac{\nu_{n+1-i}-\lambda_{i+1}+\frac{1}{2}+[\xi_{i+1}+\eta_{n+1-i}]}{2})$ in the normalization of $K_{\xi,\lambda}^{\eta,\nu}$. As explained below, this argument can be extended to all pairs $(\lambda_i,\nu_j)$.

Consider the renormalized kernel
\begin{equation}
    \mathbf{K}_{\xi,\lambda}^{\eta,\nu}=\frac{K_{\xi,\lambda}^{\eta,\nu}}{\prod_{j=1}^n L(\tfrac{1}{2},\chi_1\psi_j^{-1})\dots L(\tfrac{1}{2},\chi_{n+1-j}\psi_j^{-1})L(\tfrac{1}{2},\chi_{n+2-j}^{-1}\psi_j)\dots L(\tfrac{1}{2},\chi_{n+1}^{-1}\psi_{j}),},\label{eq:DefinitionRenormalizedKernel}
\end{equation}
which, by the duplication formula, is equal to the product of $\mathbb{K}_{\xi,\lambda}^{\eta,\nu}$ and
\begin{equation}\label{faktor}
\prod\limits_{j=1}^n\Big[\prod\limits_{i=1}^{n+1-j}\Gamma\Big(\frac{\lambda_i-\nu_j+\frac{3}{2}-[\xi_i+\eta_j]}{2}\Big)\Big]\Big [\prod\limits_{i=n+2-j}^{n+1}\Gamma\Big(\frac{\nu_j-\lambda_i+\frac{3}{2}-[\xi_i+\eta_j]}{2}\Big)\Big]
\end{equation}
and some power of $\pi$ and $2$.
\begin{theorem} \label{main thm}
The renormalized kernel $\mathbf{K}_{\xi,\lambda}^{\eta,\nu}$ extends analytically to a family of distributions in $\calD'(G)$ that depends holomorphically on $(\lambda,\nu)\in \CC^{n+1}\times \CC^n$. Moreover, $\mathbf{K}_{\xi,\lambda}^{\eta,\nu}\in \calD'(G)_{\xi,\lambda}^{\eta,\nu}$ for all $\xi$, $\eta$, $\lambda$ and $\nu$.
\end{theorem}
\begin{proof}
To prove this, it suffices to show $\mathbb{K}_{\xi,\lambda}^{\eta,\nu}$ vanishes at all the poles from (\ref{faktor}). We argue as above, but instead of using Remark \ref{alternativ beskrivelse} we apply Corollary~\ref{grimt koro}. The functions $b_i^{\pm}(\xi,\lambda,\eta,\nu)$ in Corollary~\ref{grimt koro}, where $i$ runs from $1$ to $n$, contain all the different gamma factors appearing in the new normalization.
\end{proof}

For later purpose, we also translate the Bernstein--Sato identities to the renormalized kernels $\mathbf{K}_{\xi,\lambda}^{\eta,\nu}$.

\begin{corollary}\label{cor:BSidentitiesForBFK}
    The renormalized kernel $\mathbf{K}_{\xi,\lambda}^{\eta,\nu}$ satisfies the following Bernstein--Sato identities:
    \begin{align*}
        \calD_i(\lambda)\mathbf{K}_{\xi,\lambda}^{\eta,\nu} ={}& \alpha_i(\xi,\lambda,\eta,\nu)\prod_{\substack{j=1\\\xi_{i+1}+\eta_j=0}}^n(\lambda_{i+1}-\nu_j+\tfrac{1}{2})\cdot\mathbf{K}_{\xi+e_{i+1},\lambda+e_{i+1}}^{\eta,\nu},\\
        \calP_i(\lambda,\nu)\mathbf{K}_{\xi,\lambda}^{\eta,\nu} ={}& \beta_i(\xi,\lambda,\eta,\nu)\prod_{\substack{j=1\\ \xi_{i+1}+\eta_j=0}}^{n-i}(\lambda_{i+1}-\nu_{j}+\tfrac{1}{2})\\
        & \qquad \qquad \qquad \qquad \qquad \times \prod_{\substack{k=i+2\\ \eta_{n+1-i}+\xi_k=0}}^{n+1}(\nu_{n+1-i}-\lambda_k+\tfrac{1}{2})\cdot \mathbf{K}_{\xi+e_{i+1},\lambda+e_{i+1}}^{\eta+e_{n+1-i},\nu+e_{n+1-i}},\\
        \calL_i(\lambda,\nu)\mathbf{K}_{\xi,\lambda}^{\eta,\nu} ={}& \gamma_i(\xi,\lambda,\eta,\nu)\prod_{\substack{j=1\\ \xi_i+\eta_j=0}}^{n-1-i}(\lambda_i-\nu_j+\tfrac{1}{2})\\
        &\qquad\qquad\qquad \qquad\qquad\times \prod_{\substack{k=i+2\\ \eta_{n-i}+\xi_k=0}}^{n+1}(\nu_{n-i}-\lambda_k+\tfrac{1}{2})\cdot \mathbf{K}_{\xi+e_{i+1},\lambda+e_{i+1}}^{\eta+e_{n-i},\nu+e_{n-i}},
    \end{align*}
    where $\alpha_i$, $\beta_j$ and $\gamma_i$ are nowhere vanishing entire functions of $(\lambda,\nu)$.
\end{corollary}

\begin{proof}
    The kernel $\mathbb{K}_{\xi,\lambda}^{\eta,\nu}$ satisfies the Bernstein--Sato identities \eqref{eq:BSidentitiesForBBK}. Since $\mathbf{K}_{\xi,\lambda}^{\eta,\nu}$ is the product of $\mathbb{K}_{\xi,\lambda}^{\eta,\nu}$ with \eqref{faktor} and a holomorphic and nowhere vanishing function in $(\lambda,\nu)$ (powers of $\pi$ and $2$), it suffices to compute ratios of \eqref{faktor} for parameters $(\xi,\lambda,\eta,\nu)$ that are related by the Bernstein--Sato identities. This is an easy though longish computation which we leave to the reader.
\end{proof}

\begin{remark}\label{normaliseret funktionalligning}
	With the new normalization, the functional identities in Theorem \ref{main theorem} can be written as
	\[
	\mathbf{T}_{\xi,-\lambda}^{w_i}\mathbf{K}_{\xi,\lambda}^{\eta,\nu}=\frac{\sqrt{\pi}(-1)^{(\xi_i+\xi_{i+1})(\eta_{n+1-i}+1)+\xi_i\xi_{i+1}}}{L(1,\chi_i\chi_{i+1}^{-1})}\mathbf{K}_{w_i(\xi,\lambda)}^{\eta,\nu}
	\]
	and
	\[
	\mathbf{S}_{\eta,\nu}^{w_i}\mathbf{K}_{\xi,\lambda}^{\eta,\nu}=\frac{\sqrt{\pi}(-1)^{(\eta_i+\eta_{i+1})(\xi_{n+1-i}+1)+\eta_i\eta_{i+1}}}{L(1,\psi_i^{-1}\psi_{i+1})}\mathbf{K}_{\xi,\lambda}^{w_i(\eta,\nu)}.
	\]
\end{remark}

\subsection{Optimal normalization}

We now show that the normalization we have used for $\fedK_{\xi,\lambda}^{\eta,\nu}$ is optimal in the sense that its zeros are of codimension two in $(\lambda,\nu)\in \CC^{n+1}\times \CC^n$. It is shown in the next section that there are indeed hyperplanes of codimension two on which the kernel vanishes.

\begin{theorem}\label{optimal normalization}
    For all $\xi$ and $\eta$, the set of $(\lambda,\nu)\in\CC^{n+1}\times\CC^n$ such that $\mathbf{K}_{\xi,\lambda}^{\eta,\nu}=0$ is of codimension $\geq2$ in $\CC^{n+1}\times\CC^n$.
\end{theorem}

\begin{proof}
We first make a general observation. Since the restriction of the unnormalized kernel $K_{\xi,\lambda}^{\eta,\nu}$ to the open dense subset $\{\Phi_1,\ldots,\Phi_{n+1},\Psi_1,\ldots,\Psi_n\neq0\}$ is a nowhere vanishing smooth function, the normalized kernel $\mathbf{K}_{\xi,\lambda}^{\eta,\nu}$ can only be zero if the normalization has a pole. Thus, looking at the gamma factors in the normalization \eqref{eq:DefinitionRenormalizedKernel}, either $\lambda_i-\nu_j+[\xi_i+\eta_j]+\frac{1}{2}\in -2\NN_0$ (for some $i,j$ with $i+j\leq n+1$) or $\nu_j-\lambda_i+[\xi_i+\eta_j]+\frac{1}{2}\in -2\NN_0$ (for some $i,j$ with $i+j\geq n+2$) must be satisfied in order for $\fedK_{\xi,\lambda}^{\eta,\nu}$ to vanish. Further note that if $\mathbf{K}_{\xi,\lambda}^{\eta,\nu}=0$ for $(\lambda,\nu)$ in an open set of a hyperplane of the above form, then it is zero for all $(\lambda,\nu)$ in this hyperplane by analytic continuation.\\
So let us assume that $\mathbf{K}_{\xi,\lambda}^{\eta,\nu}=0$ for all parameters on the hyperplane $\lambda_i-\nu_j+[\xi_i+\eta_j]+\frac{1}{2}=-2m$ for some $i,j$ with $i+j\leq n+1$ and some $m\in\NN_0$. Since the coefficients on the right hand side of the functional identities in Remark~\ref{normaliseret funktionalligning} are generically non-zero on this hyperplane, we can conclude that both $\mathbf{K}_{w_i(\xi,\lambda)}^{\eta,\nu}=0$ and $\mathbf{K}_{\xi,\lambda}^{w_i(\eta,\nu)}=0$ on the same hyperplane. Iterating this argument shows that $\mathbf{K}_{\xi,\lambda}^{\eta,\nu}=0$ on all hyperplanes of the form $\lambda_i-\nu_j+[\xi_i+\eta_j]+\frac{1}{2}=-2m$ with $i=1,\ldots,n+1$ and $j=1,\ldots,n$. But for $i,j$ with $i+j\geq n+2$, this hyperplane is not among the possible zeros of $\mathbf{K}_{\xi,\lambda}^{\eta,\nu}$ identified above. This shows that $\mathbf{K}_{\xi,\lambda}^{\eta,\nu}$ cannot vanish on any of the hyperplanes $\lambda_i-\nu_j+[\xi_i+\eta_j]+\frac{1}{2}=-2m$.\\
The same argument applies to the hyperplanes $\nu_j-\lambda_i+[\xi_i+\eta_j]+\frac{1}{2}=-2m$ and the proof is complete.
\end{proof}

\section{On zeros and residues}
\noindent In this section we prove some results about the vanishing and support of the distributions $\mathbf{K}_{\xi,\lambda}^{\eta,\nu}\in\calD'(G)$. Note that the support of $\mathbf{K}_{\xi,\lambda}^{\eta,\nu}$ is equal to $G$ unless $(\lambda,\nu)$ is a pole of one of the gamma factors with which we have normalized. Therefore, the value of $\mathbf{K}_{\xi,\lambda}^{\eta,\nu}$ at such a point can be viewed as a residue of the unnormalized family $K_{\xi,\lambda}^{\eta,\nu}$.

\subsection{The support at singular parameters}
Since $\mathbf{K}_{\xi,\lambda}^{\eta,\nu}\in\calD'(G)$ is equivariant under the right action of $P_G$ and under the left action of $P_H$, its support is a union of double cosets in $P_H\backslash G/P_G$. Since $(G,H)$ is a strongly spherical pair (see e.g. \cite{F23}), the latter set of double cosets is finite. While it seems to be a difficult task to determine the support of the distribution $\mathbf{K}_{\xi,\lambda}^{\eta,\nu}$ precisely, some soft arguments only involving the fact that our normalization is holomorphic yield an upper bound in terms of the zero sets of $\Phi_i$ and $\Psi_j$. Note that since both $\Phi_i$ and $\Psi_j$ are equivariant under the actions of $P_G$ and $P_H$, their zero sets are unions of double cosets in $P_H\backslash G/P_G$.

\begin{remark}\label{rem:RenormalizedKernelMultipliedWithPhiPsi}
    We first translate the identities $\Phi_kK_{\delta,s}^{\varepsilon,t}=K_{\delta+e_k,s+e_k}^{\varepsilon,t}$ and $\Psi_\ell K_{\delta,s}^{\varepsilon,t}=K_{\delta,s}^{\varepsilon+e_\ell,t+e_\ell}$ to the normalized kernels:
    \begin{align*}
        \Phi_k\mathbf{K}_{\delta,s}^{\varepsilon,t} &= \sigma_k(\xi,\lambda,\eta,\nu)\prod_{\substack{i\leq k\leq n+1-j\\\xi_i+\eta_j=0}}(\lambda_i-\nu_j+\tfrac{1}{2})\prod_{\substack{n+2-j\leq k\leq i-1\\\xi_i+\eta_j=0}}(\nu_j-\lambda_i+\tfrac{1}{2})\cdot \mathbf{K}_{\delta+e_k,s+e_k}^{\varepsilon,t},\\
        \Psi_\ell\mathbf{K}_{\delta,s}^{\varepsilon,t} &= \kappa_\ell(\xi,\lambda,\eta,\nu)\prod_{\substack{i\leq\ell\leq n-j\\\xi_i+\eta_j=0}}(\lambda_i-\nu_j+\tfrac{1}{2})\prod_{\substack{n+1-j\leq\ell\leq i-1\\\xi_i+\eta_j=0}}(\nu_j-\lambda_i+\tfrac{1}{2})\cdot \mathbf{K}_{\delta,s}^{\varepsilon+e_\ell,t+e_\ell},
    \end{align*}
    where $\sigma_k$ and $\kappa_\ell$ are nowhere vanishing holomorphic functions of $(\lambda,\nu)\in\CC^{n+1}\times\CC^n$.
\end{remark}

\begin{proposition}\label{prop:SupportOfK}
If $1\leq i\leq n+1-j\leq n$ and $\lambda_i-\nu_j+\frac{1}{2}+[\xi_i+\eta_j]\in-2\NN_0$, then the support of $\mathbf{K}_{\xi,\lambda}^{\eta,\nu}$ is contained in 
\[
\{\Phi_i=0\}\cap \{\Psi_i=0\}\cap\{\Phi_{i+1}=0\}\cap\dots\cap\{\Psi_{n-j}=0\}\cap\{\Phi_{n+1-j}=0\}.
\]
If $1\leq n+1-j\leq i-1\leq n$ and $\nu_j-\lambda_i+\frac{1}{2}+[\eta_j+\xi_i]\in-2\NN_0$, then the support of $\mathbf{K}_{\xi,\lambda}^{\eta,\nu}$ is contained in 
\[
\{\Psi_{n+1-j}=0\}\cap\{\Phi_{n+2-j}=0\}\cap\dots\cap\{\Phi_{i-1}=0\}\cap\{\Psi_{i-1}=0\}.
\] 
\end{proposition}

\begin{proof}
We only prove the first statement, the second one is shown analogously. Assume that $1\leq i\leq n+1-j\leq n$ and $\lambda_i-\nu_j+\frac{1}{2}=-m$ with $m\in\NN_0$, $m\equiv\xi_i+\eta_j\mod2$. We show by induction on $m$ that $\Phi_k^{m+1}\mathbf{K}_{\xi,\lambda}^{\eta,\nu}=0$ for any $k=i,\ldots,n+1-j$, the corresponding statement for $\Psi_\ell^m\mathbf{K}_{\xi,\lambda}^{\eta,\nu}$ is treated similarly. For $m=0$ this follows directly from the first identity in Remark~\ref{rem:RenormalizedKernelMultipliedWithPhiPsi} since the linear factor $\lambda_i-\nu_j+\frac{1}{2}$ appears on the right hand side. Now assume $m>0$, then $\Phi_k\mathbf{K}_{\xi,\lambda}^{\eta,\nu}$ is a multiple of $\mathbf{K}_{\xi',\lambda'}^{\eta',\nu'}$, where in particular $\xi'_i=\xi_i+1$, $\lambda'_i=\lambda_i+1$ and $\eta_j'=\eta_j$, $\nu_j'=\nu_j$ by Remark~\ref{rem:RenormalizedKernelMultipliedWithPhiPsi}. Since $\lambda_i'-\nu_j'+\frac{1}{2}=-(m-1)$ and $m-1\equiv\xi_i'+\eta_j'\mod2$, the induction assumption implies that $\Phi_k^m\mathbf{K}_{\xi',\lambda'}^{\eta',\nu'}=0$, hence $\Phi_k^{m+1}\mathbf{K}_{\xi,\lambda}^{\eta,\nu}=0$ as claimed.
\end{proof}
\subsection{Some zeros}

The following result identifies a locally finite union of affine subspaces of codimension two in the space of parameters $(\lambda,\nu)\in\CC^{n+1}\times\CC^n$ for which the holomorphic family of distributions $\mathbf{K}_{\xi,\lambda}^{\eta,\nu}$ vanishes. 

\begin{theorem}\label{Alle nuller}
The distribution $\mathbf{K}_{\xi,\lambda}^{\eta,\nu}$ vanishes for all parameters $(\xi,\lambda,\eta,\nu)$ in the sets
\[
\mathcal{N}_{i,j,k}=\{\lambda_i-\nu_k+\tfrac{1}{2}+[\xi_i+\eta_k]\in -2\NN_0\}\cap \{\nu_k-\lambda_j+\tfrac{1}{2}+[\eta_k+\xi_j]\in-2\NN_0\},
\]
where $i<j$ and $k\in\{1,\dots,n\}$ and 
\[
\mathcal{M}_{i,j,k}=\{\nu_j-\lambda_k+\tfrac{1}{2}+[\eta_j+\xi_k]\in-2\NN_0\}\cap \{\lambda_k-\nu_i+\tfrac{1}{2}+[\xi_k+\eta_i]\in -2\NN_0\},
\]
where $i<j$ and $k\in \{1,\dots,n+1\}$.
\end{theorem}

We prove this theorem in three steps.

\begin{lemma}\label{first zero}
For $(\varepsilon_1,t_1)=(0,-1)$ and $\Re(s_i)\geq 0$ ($i=1,\ldots,n$), $\Re(t_j)\geq 0$ ($j=2,\dots,n$) the kernel $\mathbf{K}_{\delta,s}^{\varepsilon,t}$ is the distribution
\[
    \mathbf{K}_{\delta,s}^{\varepsilon,t}(g)=(-1)^{\varepsilon_2}\frac{\delta(g_{n,1})|g_{n+1,1}|_{\delta_1+\delta_2}^{s_1+s_2}|g_{n,2}|_{\delta_2+\varepsilon_2}^{s_2+t_2}|g_{n-1,1}|_{\varepsilon_2}^{t_2}\prod_{i=3}^{n+1}|\Phi_i(g)|_{\delta_i}^{s_i}\prod_{j=3}^{n}|\Psi_j(g)|_{\varepsilon_j}^{t_j}}{n'(\delta,s,\varepsilon,t)},
\]
where $n'(\delta,s,\varepsilon,t)=(n(\delta,s,\varepsilon,t)/\Gamma(\frac{t_1+1}{2}))\vert_{t_1=-1}$. Furthermore, $\mathbf{K}_{\delta,s}^{\varepsilon,t}=0$ if additionally $s_1+[\delta_1]+1\in -2\NN_0$ or $s_2+[\delta_2]+1\in -2\NN_0$.
\end{lemma}

\begin{proof}
For the parameters $(s,t)$ considered above, both $|\Phi_i|_{\delta_i}^{s_i}$ ($i=1,\ldots,n+1$) and $|\Psi_j|_{\varepsilon_j}^{t_j}$ ($j=2,\ldots,n$) are continuous functions on $G$ which we view as an open subset of all $(n+1)\times(n+1)$ matrices over $\RR$. The residue of the Riesz distribution $|\Psi_1(g)|_{\varepsilon_1}^{t_1}=|g_{n,1}|^{t_1}$ at $t_1=-1$ is given by (see e.g. \cite[Chapter 3.3]{GS64})
$$ \left.\Gamma\Big(\frac{t_1+1}{2}\Big)^{-1}|\Psi_1(g)|_{\varepsilon_1}^{t_1}\right|_{t_1=-1}=\delta(g_{n,1}), $$
and as a functional on $C_c^\infty(G)$ it extends continuously to $C(G)$. Moreover, the delta distribution satisfies $\varphi(x)\delta(x)=\varphi(0)\delta(x)$ for every $\varphi\in C(\RR)$. Therefore, under the above assumptions on the parameters $(s,t)$, we can simply take the residue of $|\Psi_1(g)|_{\varepsilon_1}^{t_1}$ and multiply it with the remaining factors $|\Phi_i|_{\delta_i}^{s_i}$ and $|\Psi_j|_{\varepsilon_j}^{t_j}$ specialized to $g_{n,1}=0$. This shows the claimed formula.\\
To show the second statement, note that the normalization $n(\delta,s,\varepsilon,t)$ contains the gamma factors $\Gamma(\frac{s_1+[\delta_1]+1}{2})$ and $\Gamma(\frac{s_2+[\delta_2]+1}{2})$. If $s_1+[\delta_1]+1=-2m\in-2\NN_0$ and $\Re(s_2)\geq2m+[\delta_1]+1$, then the function $|g_{n+1,1}|^{s_1+s_2}_{\delta_1+\delta_2}$ is still continuous in $g$ and so are the other factors that are multiplied with $\delta(g_{n,1})$ in the numerator of the formula above. But since $n'(\delta,s,\varepsilon,t)^{-1}=0$ we obtain $\mathbf{K}_{\delta,s}^{\varepsilon,t}=0$. As $\mathbf{K}_{\delta,s}^{\varepsilon,t}$ is entire in $(s_2,\ldots,s_{n+1},t_2,\ldots,t_n)\in\CC^{2n-1}$ (with fixed $t_1=-1$ and $s_1=-1-[\delta_1]-2m$) and vanishes on the open set of parameters with positive real part, it is identically zero. The case $s_2+[\delta_2]+1\in-2\NN_0$ follows in a similar way.
\end{proof}

We now use the Bernstein--Sato identities to extend the condition $(\varepsilon_1,t_1)=(0,-1)$ to $t_1+[\varepsilon_1]+1\in-2\NN_0$. For this we use the Bernstein--Sato identity for $\calD_1$ (see Corollary~\ref{cor:BSidentitiesForBFK}) written in the spectral parameters and in a shifted version:
\begin{equation}\label{shift}
    \calD_1(s-e_2,t+e_1)\mathbf{K}_{\delta-e_2,s-e_2}^{\varepsilon+e_1,t+e_1} = p_1(\delta_2,s_2)p_2(\varepsilon_1,t_1)p_3(\delta,s,\varepsilon,t)\cdot\mathbf{K}_{\delta,s}^{\varepsilon,t},
\end{equation}
where
\begin{align*}
    p_1(\delta_2,s_2) &= \begin{cases}1&\mbox{if $\delta_2=0$,}\\s_2&\mbox{if $\delta_2=1$,}\end{cases} \qquad\qquad
    p_2(\varepsilon_1,t_1) = \begin{cases}1&\mbox{if $\varepsilon_1=0$,}\\t_1+1&\mbox{if $\varepsilon_1=1$,}\end{cases},\\
    p_3(\delta,s,\varepsilon,t) &= \const\times\prod_{\substack{j=1\\\delta_2+\varepsilon_2+\cdots+\delta_{n+1-j}=1}}^{n-2}(s_2+t_2+\cdots+s_{n+1-j}+n-j-1).
\end{align*}

\begin{lemma}\label{first zeros non-spherical}
Let $t_1+[\varepsilon_1]+1=-2k$, then $\mathbf{K}_{\delta,s}^{\varepsilon,t}=0$ if $s_1+[\delta_1]+1\in-2\NN_0$ or $s_2+[\delta_2]+1\in-2\NN_0$. 
\end{lemma}

\begin{proof}
We write $t_1=-1-m$ with $m\equiv\varepsilon_1\mod2$ and proceed by induction on $m$, the base case $m=0$ being Lemma~\ref{first zero}. Assume that the statement holds for $m$ and let $t_1=-1-(m+1)$ with $m\geq0$. In particular, $p_2(\varepsilon_1,t_1)\neq0$. Assume that either $s_1+[\delta_1]+1\in-2\NN_0$ or $s_2+[\delta_2]+1\in-2\NN_0$. In the latter case we also have $(s_2-1)+[\delta_2-1]+1\in-2\NN_0$, so that in both case $\mathbf{K}_{\delta-e_2,s-e_2}^{\varepsilon+e_1,t+e_1}=0$ by induction assumption. This implies that the left hand side of \eqref{shift} vanishes, and since $p_1(\delta_2,s_2)\neq0$, $p_2(\varepsilon_1,t_1)\neq0$ and for generic $(s_3,\ldots,s_{n+1},t_2,\ldots,t_n)$ also $p_3(\delta,s,\varepsilon,t)\neq0$, we obtain $\mathbf{K}_{\delta,s}^{\varepsilon,t}=0$ generically, hence everywhere by continuity. This shows the claim.
\end{proof}

\begin{proof}[Proof of Theorem~\ref{Alle nuller}]
The zeros we found in Lemma \ref{first zeros non-spherical} in $(\xi,\lambda,\eta,\nu)$-coordinates are 
        \begin{align*}
    \mathcal{N}_{1,2,n}&=\{\lambda_1-\nu_n+\tfrac{1}{2}+[\xi_1+\eta_n]\in-2\NN_0\}\cap \{\nu_n-\lambda_2+\tfrac{1}{2}+[\eta_n+\xi_2]\in-2\NN_0\},\\
    \mathcal{M}_{n-1,n,2}&=\{\nu_n-\lambda_2+\tfrac{1}{2}+[\eta_n+\xi_2]\in-2\NN_0\}\cap \{ \lambda_2-\nu_{n-1}+\tfrac{1}{2}+[\xi_2+\eta_{n-1}]\in-2\NN_0\}.
    \end{align*}
Consider the two functional identities 
\begin{align*}
    \mathbf{A}_{\xi,\lambda}^{\eta,\nu}\circ \mathbf{T}^{w_i}_{w_i(\xi,\lambda)}&=\frac{c}{\Gamma(\frac{\lambda_i-\lambda_{i+1}+1+[\xi_i+\xi_{i+1}]}{2})}\mathbf{A}_{w_i(\xi,\lambda)}^{\eta,\nu},\\
    \mathbf{T}_{\eta,\nu}^{w_i}\circ \mathbf{A}_{\xi,\lambda}^{\eta,\nu}&=\frac{c'}{\Gamma(\frac{\nu_{i+1}-\nu_i+1+[\eta_{i+1}+\eta_i]}{2})}\mathbf{A}_{\xi,\lambda}^{w_i(\eta,\nu)},
\end{align*}
for some constants $c$, $c'$ that depend only on $(\xi,\eta)$. Using the first one for $i=2$, we can for $\lambda_2-\lambda_3+1+[\xi_2+\xi_3]\notin -2\NN_0$ conclude that the zeros of $\mathbf{A}_{\xi,\lambda}^{\eta,\nu}$ from $\mathcal{N}_{1,2,n}$ are also zeros for $\mathbf{A}_{w_2(\xi,\lambda)}^{\eta,\nu}$. This implies that $\mathcal{N}_{1,3,n}\cap\{\lambda_3-\lambda_2+1+[\xi_2+\xi_3]\notin -2\NN_0\}$ are zeros for $\mathbf{A}_{\xi,\lambda}^{\eta,\nu}$. Since $\mathcal{N}_{1,3,n}\cap\{\lambda_3-\lambda_2+1+[\xi_2+\xi_3]\notin -2\NN_0\}$ is dense in $\mathcal{N}_{1,3,n}$, it follows that $\mathbf{A}_{\xi,\lambda}^{\eta,\nu}$ vanishes on all of $\mathcal{N}_{1,3,n}$. Applying this procedure with both functional identities and all permutations $w_i$, we get all the vanishing sets from Theorem \ref{Alle nuller}. Note that using the first identity for $i=1$ would not provide any new zeros as the gamma factor vanishes in the functional identity on all of $\mathcal{N}_{1,2,n}$, this is the reason that we keep $i<j$ in the statement of the Theorem.
\end{proof}

\section{Relation to Rankin--Selberg integrals and evaluation on the spherical vector}\label{sec:EvaluationSphericalVector}

\noindent In this final section we relate the symmetry breaking operators $\fedA_{\xi,\lambda}^{\eta,\nu}$ with distribution kernels $\fedK_{\xi,\lambda}^{\eta,\nu}$ to the local archimedean Rankin--Selberg integrals involving the Whittaker models of $\pi_{\xi,\lambda}$ and $\tau_{\eta,\nu}$. This is done by combining results of Li, Liu, Su and Sun~\cite{LLSS23} with a certain integral formula for the open $H$-orbit in $G/P_G\times H/P_H$. Since the evaluation of Rankin--Selberg integrals on spherical Whittaker vectors was carried out by Ishii and Stade~\cite{IS13}, we obtain a formula for the evaluation of our operators $\fedA_{\xi,\lambda}^{\eta,\nu}$ for $\xi=(0,\ldots,0)$ and $\eta=(0,\ldots,0)$ on the $K_G$-fixed vector $\mathbf{1}_\lambda\in \pi_\lambda=\pi_{\xi,\lambda}$ as a byproduct. For simplicity, everything in this chapter is computed up to a non-zero constant which is independent of $\lambda$ and $\nu$.

\subsection{Symmetry breaking operators vs. invariant forms}

We first relate the operator $\mathbf{A}_{\xi,\lambda}^{\eta,\nu}$, or rather its unnormalized counterpart $A_{\xi,\lambda}^{\eta,\nu}$, to an invariant form on the tensor product representation $\pi_{\xi,\lambda}\otimes\tau_{\eta,-\nu}$. The non-degenerate invariant form $\tau_{\eta,\nu}\otimes\tau_{\eta,-\nu}\to\CC$ from \eqref{eq:InvariantPairing} gives rise to a natural isomorphism (see e.g. \cite[Proposition 5.3]{KS18})
$$ \Hom_H(\pi_{\xi,\lambda}|_H,\tau_{\eta,\nu})\to\Hom_H(\pi_{\xi,\lambda}|_H\otimes\tau_{\eta,-\nu},\CC). $$
In this way, we obtain a holomorphic family of invariant forms
$$ \operatorname{SBO}_{\xi,\lambda}^{\eta,\nu}:\pi_{\xi,\lambda}|_H\otimes\tau_{\eta,-\nu}\to\CC, \quad f_1\otimes f_2\mapsto\int_{H/P_H}A_{\xi,\lambda}^{\eta,\nu}f_1(h)f_2(h)\,d(hP_H). $$
In the next section we relate this invariant form to another invariant form which is defined using Whittaker models and Rankin--Selberg integrals.

\subsection{Whittaker models and Rankin--Selberg integrals}

We fix the non-degenerate unitary character
$$ \psi_G:N_G\to\CC^\times, \quad \psi_G(g) = e^{2\pi i(g_{1,2}+\cdots+g_{n,n+1})} $$
of $N_G$ and consider the space
$$ C^\infty(G/N_G,\psi_G) = \{u\in C^\infty(G):u(gn)=\psi_G(n)^{-1}u(g)\mbox{ for all }g\in G,n\in N_G\}. $$
For every $(\xi,\lambda)\in(\ZZ/2\ZZ)^{n+1}\times\CC^{n+1}$ there is a unique (up to scalar multiples) continuous intertwining operator
$$ W^{\xi,\lambda}:\pi_{\xi,\lambda}\to C^\infty(G/N_G,\psi_G), \quad f\mapsto W^{\xi,\lambda}_f. $$
We normalize $W^{\xi,\lambda}$ by the property that
$$ W^{\xi,\lambda}_f(x) = \int_{N_G}f(xnw_0^G)\overline{\psi_G(n)}\,dn \qquad (x\in G) $$
for every $f$ for which the integral converges, where $w_0^G$ denotes the longest Weyl group element. The image of $W^{\xi,\lambda}$ is called \emph{Whittaker model} for $\pi_{\xi,\lambda}$.

Using the character $\psi_H=\overline{\psi_G}|_{N_H}$ of $N_H$, we can define a similar Whittaker model for principal series of $H$:
$$ \overline{W}^{\eta,\nu}:\tau_{\eta,\nu}\to C^\infty(H/N_H,\psi_H), \quad f\mapsto\overline{W}^{\eta,\nu}_f, \quad \overline{W}^{\eta,\nu}_f(x) = \int_{N_H}f(xnw_0^H)\,dn. $$

The archimedean local Rankin--Selberg integrals construct natural $H$-invariant forms on the tensor products of the Whittaker models of $\pi_{\xi,\lambda}|_H$ and $\tau_{\eta,\nu}$. Composing them with the maps $W^{\xi,\lambda}$ and $\overline{W}^{\eta,-\nu}$, we obtain a family of $H$-invariant forms
$$ \operatorname{RS}_{\xi,\lambda}^{\eta,\nu}:\pi_{\xi,\lambda}|_H\otimes\tau_{\eta,-\nu}\to\CC, \quad f_1\otimes f_2\mapsto\int_{H/N_H}W^{\xi,\lambda}_{f_1}(h)\overline{W}^{\eta,-\nu}_{f_2}(h)\,d(hN_H), $$
where the integral converges absolutely for sufficiently positive $\lambda$ and $-\nu$ and is extended meromorphically to all $(\lambda,\nu)\in\CC^{n+1}\times\CC^n$.

\subsection{Comparing invariant forms}
In this section we explicitly relate the invariant forms $\operatorname{SBO}_{\xi,\lambda}^{\eta,\nu}$ and $\operatorname{RS}_{\xi,\lambda}^{\eta,\nu}$. The result is most easily stated in terms of the characters $\chi_i(x)=|x|_{\xi_i}^{\lambda_i}$ and $\psi_j(x)=|x|_{\eta_j}^{\nu_j}$ of $\RR^\times$ and their corresponding local archimedean $L$-factors:

\begin{proposition}\label{prop:ComparisonSBOvsRS}
    For all $(\xi,\eta)\in(\ZZ/2\ZZ)^{n+1}\times(\ZZ/2\ZZ)^n$ the following holds as identity of meromorphic functions in $(\lambda,\nu)\in\CC^{n+1}\times\CC^n$:
    $$ \operatorname{SBO}_\lambda^\nu = \const\times\prod_{i+j\geq n+2}\frac{L(\frac{1}{2},\chi_i^{-1}\psi_j)}{L(\frac{1}{2},\chi_i\psi_j^{-1})}\operatorname{RS}_\lambda^\nu. $$
\end{proposition}

This result is proven by comparing both $\operatorname{SBO}_{\xi,\lambda}^{\eta,\nu}$ and $\operatorname{RS}_{\xi,\lambda}^{\eta,\nu}$ to yet another $H$-invariant form (see Theorem~\ref{thm:LLSS} and Proposition~\ref{prop:ComparisonOvsSBO}). It is simply given by integrating over the open dense $H$-orbit in $G/P_G\times H/P_H$, where $H$ is acting diagonally. Note that by \cite[Lemma 6.3]{F23} there is a unique open double coset in $P_H\backslash G/P_G$. By means of the diffeomorphism $\diag(H)\backslash(G\times H)\to G,\,\diag(H)(g,h)\mapsto h^{-1}g$, this implies that there is a unique open $H$-orbit in $G/P_G\times H/P_H$. Let $(z_1,z_2)\in G\times H$ be a representative of this open $H$-orbit and put
$$ \calO_{\xi,\lambda}^{\eta,\nu}:\pi_{\xi,\lambda}\otimes\tau_{\eta,-\nu}\to\CC, \quad f_1\otimes f_2\mapsto\int_H f_1(hz_1)f_2(hz_2)\,dh. $$
This integral converges absolutely for $(\lambda,\nu)$ in some non-empty open domain of $\CC^{n+1}\times\CC^n$ (see \cite[Proposition 1.4]{LLSS23}). The main result of \cite{LLSS23} is a comparison between $\calO_{\xi,\lambda}^{\eta,\nu}$ and $\operatorname{RS}_{\xi,\lambda}^{\eta,\nu}$ (which also implies that $\calO_{\xi,\lambda}^{\eta,\nu}$ extends meromorphically to all $(\lambda,\nu)\in\CC^{n+1}\times\CC^n$):

\begin{theorem}[{see \cite[Theorem 1.6~(b)]{LLSS23}}]\label{thm:LLSS}
    For all $(\xi,\eta)\in(\ZZ/2\ZZ)^{n+1}\times(\ZZ/2\ZZ)^n$ the following holds as identity of meromorphic functions in $(\lambda,\nu)\in\CC^{n+1}\times\CC^n$:
    $$ \calO_{\xi,\lambda}^{\eta,\nu} = \const\times\prod_{i+j\geq n+2}\frac{L(\frac{1}{2},\chi_i^{-1}\psi_j)}{L(\frac{1}{2},\chi_i\psi_j^{-1})}\operatorname{RS}_{\xi,\lambda}^{\eta,\nu}. $$
\end{theorem}

(Note that \cite{LLSS23} is using principal series induced from the opposite parabolic subgroups $\overline{P}_G$ and $\overline{P}_H$, so to translate their result one has to conjugate by the longest Weyl group elements $w_0^G$ and $w_0^H$.)

In order to prove Proposition~\ref{prop:ComparisonSBOvsRS} it therefore remains to compare the invariant forms $\operatorname{SBO}_{\xi,\lambda}^{\eta,\nu}$ defined by the symmetry breaking operators $A_{\xi,\lambda}^{\eta,\nu}$ to the forms $\calO_{\xi,\lambda}^{\eta,\nu}$.

\begin{proposition}\label{prop:ComparisonOvsSBO}
    For all $(\xi,\eta)\in(\ZZ/2\ZZ)^{n+1}\times(\ZZ/2\ZZ)^n$ the following holds as identity of meromorphic functions in $(\lambda,\nu)\in\CC^{n+1}\times\CC^n$:
    $$ \calO_{\xi,\lambda}^{\eta,\nu} = \const\times\operatorname{SBO}_{\xi,\lambda}^{\eta,\nu}. $$
\end{proposition}

The proof requires a change of variables formula which we show separately. First note that the open $H$-orbit in $G/P_G\times H/P_H$ contains the representative $(z_1,z_2)=(z_0,e)$, where $z_0\in G=\GL(n+1,\RR)$ is the representative of the open double coset in $P_H\backslash G/P_G$ with entries
$$ (z_0)_{i,j} = \begin{cases}1&\mbox{if $i+j=n+1$ or $i+j=n+2$,}\\0&\mbox{else.}\end{cases} $$
It has the property that $\Phi_i(z_0)=\Psi_j(z_0)=1$ for all $1\leq i\leq n+1$ and $1\leq j\leq n$, so $K_\lambda^\nu(z_0)=1$ for all $\lambda$ and $\nu$. Since both $\overline{N}_GP_G/P_G$ and $P_Hz_0P_G/P_G$ are open dense in $G/P_G$, the two subsets
$$ \overline{N}_G^\reg = \overline{N}_G\cap P_Hz_0P_G \subseteq \overline{N}_G \qquad \mbox{and} \qquad P_H^\reg = P_H\cap\overline{N}_GP_Gz_0^{-1} \subseteq P_H $$
of $\overline{N}_G$ and $P_H^\reg$ are both open and dense. Now, every $\overline{n}_G\in\overline{N}_G^\reg$ can be written as $\overline{n}_G=p_Hz_0p_G$ with $p_G\in P_G$ and $p_H\in P_H^\reg$. Conversely, for every $p_H\in P_H^\reg$ there exists $p_G\in P_G$ such that $\overline{n}_G=p_Hz_0p_G\in\overline{N}_G$. (In fact, both decompositions are unique since  $P_H\cap z_0P_Gz_0^{-1}=\{e\}$.) We obtain a diffeomorphism
$$ \overline{N}_G^\reg\to P_H^\reg, \quad \overline{n}_G\mapsto p_H, \qquad \mbox{where }\overline{n}_G\in p_Hz_0P_G, $$
and the following lemma computes the Jacobian of this diffeomorphism:

\begin{lemma}\label{integral lemma}
For $f\in L^1(\overline{N}_G^\reg)$ the following integral formula holds
\[
\int_{\overline{N}_G^\reg}f(\overline{n}_G)\,d\overline{n}_G=\int_{P_H^\reg} f(p_Hz_0p_G)p_G^{2\rho_G}p_H^{-2\rho_H}d^rp_H,
\]
where $p_G$ denotes the unique element in $P_G$ such that $p_Hz_0p_G\in \overline{N}_G^\reg$ and $d^rp_H$ is a suitably normalized right Haar measure on $P_H$.
\end{lemma}

\begin{proof}
Recall that there exists a unique (up to scalar multiples) $G$-invariant integral $d(gP)$ on sections of the bundle $G\times_{P_G}\CC_{2\rho_G}$ (cf. Section~\ref{sec:PrincipalSeries}). Since both $\overline{N}_GP_G/P_G\simeq\overline{N}_G$ and $P_Hz_0P_G/P_G\simeq P_H$ are open and dense in $G$ and the integral $d(gP_G)$ is $G$-invariant, it follows that the (left) Haar measures $d\overline{n}_G$ and $d^lp_H$ on $\overline{N}_G$ and $P_H$ can be normalized such that
$$ \int_{G/P_G}f(gP_G)\,d(gP_G) = \int_{\overline{N}_G}f(\overline{n}_G)\,d\overline{n}_G = \int_{P_H}f(p_Hz_0)\,d^lp_H \quad \mbox{for all }f\in L^1(G\times_{P_G}\CC_{2\rho_G}). $$
Now, given a function $f\in L^1(\overline{N}_G^\reg)$, we can extend it to a function $f\in L^1(G\times_{P_G}\CC_{2\rho_G})$ by
$$ f(g) = \begin{cases}p_G^{-2\rho_G}f(\overline{n}_G)&\mbox{for }g=\overline{n}_Gp_G\in\overline{N}_G^\reg P_G,\\0&\mbox{else.}\end{cases} $$
Then
$$ \int_{\overline{N}_G^\reg} f(\overline{n}_G)\,d\overline{n}_G = \int_{G/P_G} f(gP_G)\,d(gP_G) = \int_{P_H} f(p_Hz_0)\,d^lp_H = \int_{P_H}f(p_Hz_0)p_H^{-2\rho_H}\,d^rp_H. $$If now $p_G$ denotes the unique element in $P_G$ such that $\overline{n}_G=p_Hz_0p_G\in\overline{N}_G$, then this equals
\[ \int_{P_H^\reg}f(p_Hz_0p_G)p_G^{2\rho_G}p_H^{-2\rho_H}\,d^rp_H.\qedhere \]
\end{proof}

\begin{proof}[Proof of Proposition~\ref{prop:ComparisonOvsSBO}]
Let $f_1\in \pi_{\xi,\lambda}$ and $f_2\in \tau_{\eta,-\nu}$ and assume that $(\lambda,\nu)\in\CC^{n+1}\times\CC^n$ are such that $K_{\xi,\lambda}^{\eta,\nu}(g)$ is a locally integrable function. Using that $\overline{N}_HP_H^\reg$ is open and dense in $H$ and that the (appropriately normalized) Haar measures $dh$, $d\overline{n}_H$ and $dp_H$ on $H$, $\overline{n}_H$ and $dp_H$ satisfy $dh=d\overline{n}_H\,d^rp_H$ (see e.g. \cite[Chapter V, \S6, Consequence 5]{K16}), we find
\[
\calO_{\xi,\lambda}^{\eta,\nu}(f_1\otimes f_2) = \int_Hf_1(hz_0)f_2(h)\,dh=\int_{\overline{N}_H}\int_{P_H^\reg} f_1(\overline{n}_Hp_Hz_0)f_2(\overline{n}_H)p_H^{\eta,\nu-\rho_H}\,d^rp_H\,d\overline{n}_H,
\]
where we have used the principal series equivariance for $f_2$ in the second step, denoting by $p_H\mapsto p_H^{\eta,\nu}$ the character of $P_H$ that $\tau_{\eta,\nu}$ is induced from. Writing $K_{\xi,\lambda}^{\eta,\nu}(p_Hz_0p_G)=p_G^{\xi,\lambda-\rho_G}p_H^{\eta,\nu+\rho_H}K_\lambda^\nu(z_0)=p_G^{\xi,\lambda-\rho_G}p_H^{\eta,\nu+\rho_H}$ and using the principal series equivariance for $f_1$ gives
\[
    = \int_{\overline{N}_H}\int_{P_H^\reg}K_{\xi,\lambda}^{\eta,\nu}(p_Hz_0p_G)f_1(\overline{n}_Hp_Hz_0p_G)f_2(\overline{n}_H)p_G^{2\rho_G}p_H^{-2\rho_H}\,d^rp_H\,d\overline{n}_H.
\]
Now we can apply Lemma~\ref{integral lemma} to obtain
\begin{multline*}
    = \int_{\overline{N}_H}\int_{\overline{N}_G^\reg}K_{\xi,\lambda}^{\eta,\nu}(\overline{n}_G)f_1(\overline{n}_H\overline{n}_G)f_2(\overline{n}_H)\, d\overline{n}_Gd\overline{n}_H=\int_{\overline{N}_H}A_{\xi,\lambda}^{\eta,\nu}f_1(\overline{n}_H)f_2(\overline{n}_H)\,d\overline{n}_H\\
    = \operatorname{SBO}_{\xi,\lambda}^{\eta,\nu}(f_1\otimes f_2)
\end{multline*}
by \eqref{eq:GinvariantIntegralKNbar} and \eqref{eq:SBOasIntegral}.
\end{proof}

\subsection{Evaluation on the spherical vector}

Denote by $\pi_\lambda$ and $\tau_\nu$ the representations $\pi_{\xi,\lambda}$ and $\tau_{\eta,\nu}$ for $\xi=(0,\ldots,0)$ and $\eta=(0,\ldots,0)$ and by $\mathbf{A}_\lambda^\nu=\mathbf{A}_{\xi,\lambda}^{\eta,\nu}$ the corresponding holomorphic family of symmetry breaking operators. The representation $\pi_\lambda$ contains a unique $K_G$-fixed vector $\mathbf{1}_\lambda$ such that $\mathbf{1}_\lambda(e)=1$. Abusing notation, we denote by $\mathbf{1}_\nu$ the analogous $K_H$-fixed vector in $\tau_\nu$. Since $\mathbf{A}_\lambda^\nu\mathbf{1}_\lambda$ is $K_H$-fixed, it has to be a scalar multiple of $\mathbf{1}_\nu$:
\begin{equation}
    \mathbf{A}_\lambda^\nu\mathbf{1}_\lambda = \beta(\lambda,\nu)\mathbf{1}_\nu.\label{eq:DefBeta}
\end{equation}
The main result of this subsection is a formula for $\beta(\lambda,\nu)$. For the statement recall that $\chi_i$ and $\psi_j$ denote the characters of $\RR^\times$ given by $\chi_i(x)=|x|^{\lambda_i}$ and $\psi_i(x)=|x|^{\nu_j}$.

\begin{theorem}\label{thm:EvaluationSphericalVector}
    For all $(\lambda,\nu)\in\CC^{n+1}\times\CC^n$ we have
    $$ \mathbf{A}_\lambda^\nu\mathbf{1}_\lambda = \const\times e_G(\lambda)e_H(-\nu)\mathbf{1}_\nu, $$
    where the non-zero constant is independent of $\lambda$ and $\nu$ and
    \begin{equation}
        e_G(\lambda) = \prod_{1\leq i<j\leq n+1}L(1,\chi_i\chi_j^{-1})^{-1} \qquad \mbox{and} \qquad e_H(-\nu) = \prod_{1\leq i<j\leq n}L(1,\psi_i^{-1}\psi_j)^{-1}\label{eq:Efunction}
    \end{equation}
    are the Harish-Chandra $e$-functions for $G$ and $H$.
\end{theorem}

Note that the denominator of $e_G(\lambda)$ and $e_H(\nu)$ is the denominator in the Gindikin--Karpelevic formula for the corresponding Harish-Chandra $c$-function.

\begin{proof}
Pairing both sides of \eqref{eq:DefBeta} with $\mathbf{1}_{-\nu}\in\tau_{-\nu}$ and using that the pairing of $\mathbf{1}_\nu$ and $\mathbf{1}_{-\nu}$ is a positive constant independent of $\nu$, we find that
\begin{equation}
    \beta(\lambda,\nu) = \const\times\frac{\operatorname{SBO}_\lambda^\nu(\mathbf{1}_\lambda\otimes\mathbf{1}_{-\nu})}{\prod_{i+j\leq n+1}L(\frac{1}{2},\chi_i\psi_j^{-1})\prod_{i+j\geq n+2}L(\frac{1}{2},\chi_i^{-1}\psi_j)}.\label{eq:BetaInTermsOfSBO}
\end{equation}
By Proposition~\ref{prop:ComparisonSBOvsRS} this can be written as
\[ \beta(\lambda,\nu) = \const\times\frac{\operatorname{RS}_\lambda^\nu(\mathbf{1}_\lambda\otimes\mathbf{1}_{-\nu})}{\prod_{i=1}^{n+1}\prod_{j=1}^nL(\frac{1}{2},\chi_i\psi_j^{-1})}. \]
The claim now follows from the following theorem for $\operatorname{RS}_\lambda^\nu(\mathbf{1}_\lambda\otimes\mathbf{1}_{-\nu})$ by Ishii and Stade.
\end{proof}

\begin{theorem}[{see \cite[Theorem 3.1~(1)]{IS13}}]\label{thm:IS}
    For $(\lambda,\nu)\in\CC^{n+1}\times\CC^n$ we have
    $$ \operatorname{RS}_\lambda^\nu(\mathbf{1}_\lambda\otimes\mathbf{1}_{-\nu}) = \const\times e_G(\lambda)e_H(-\nu)\prod_{i=1}^{n+1}\prod_{j=1}^n L\Big(\frac{1}{2},\chi_i\psi_j^{-1}\Big). $$
\end{theorem}

(Note that \cite{IS13} uses $\nu$ instead of $-\nu$ and normalizes the Whittaker function $W_{\mathbf{1}_\lambda}^\lambda$ by $e_G(\lambda)^{-1}$.)

\appendix

\section{Integral Formulas}

\noindent In this appendix we state an integral formula that is used in the proof of Theorem~\ref{main theorem} and also in the discussion of the special case $n=1$ in Section~\ref{eksempel funktionalligning}.

\begin{proposition}[See {{\cite[Prop. B.2]{BD23}}}]\label{foldning}
	For $\alpha,\beta\in\CC$ with $\textup{Re}(\alpha),\textup{Re}(\beta)>-1$ and $\textup{Re}(\alpha+\beta+1)<0$ and $\varepsilon,\xi\in\{0,1\}$ we have
	$$\int_\RR |y|^\alpha_\varepsilon|x-y|^\beta_\xi dy=t(\alpha,\beta,\varepsilon,\xi)|x|^{\alpha+\beta+1}_{\varepsilon+\xi}, $$
	where 
	\begin{align*}
		t(\alpha,\beta,\varepsilon,\xi)=(-1)^{\varepsilon\xi}\sqrt{\pi}\frac{\Gamma(\frac{\alpha+1+\varepsilon}{2})\Gamma(\frac{\beta+1+\xi}{2})\Gamma(\frac{-\alpha-\beta-1+[\varepsilon+\xi]}{2})}{\Gamma(\frac{-\alpha+\varepsilon}{2})\Gamma(\frac{-\beta+\xi}{2})\Gamma(\frac{\alpha+\beta+2+[\xi+\varepsilon]}{2})}.
	\end{align*}
\end{proposition}

By a simple change of variables this implies:

\begin{corollary}\label{foldning cor}
	Let $a,c\in \RR^\times$ and either $b$ or $d$ non-zero in $\RR$. For $\Re(\alpha),\Re(\beta)>-1$ with $\Re(\alpha+\beta+1)<0$ we have
	\[
	\int_\RR |ax+b|^\alpha_\varepsilon |cx+d|_\xi^\beta dx=(-1)^{\varepsilon}t(\alpha,\beta,\varepsilon,\xi)|a|_\xi^{-\beta-1}|c|_\varepsilon^{-\alpha-1}|ad-bc|_{\varepsilon+\xi}^{\alpha+\beta+1}
	\]
    with $t(\alpha,\beta,\varepsilon,\xi)$ as in Proposition~\ref{foldning}.
\end{corollary}

\end{document}